\documentclass{amsart}
\usepackage{amssymb}
\usepackage{amsmath}
\usepackage{amsfonts}

\newtheorem{theorem}{Theorem}
\theoremstyle{plain}

\newtheorem{conclusion}{Conclusion}

\newtheorem{corollary}{Corollary}

\newtheorem{definition}{Definition}
\newtheorem{example}{Example}

\newtheorem{proposition}{Proposition}
\newtheorem{remark}{Remark}

\newtheorem{summary}{Summary}
\numberwithin{equation}{section}
\input{tcilatex}

\begin{document}
\title{Separation axioms in bi-soft Topological Spaces}
\author{Munazza Naz}
\address{Department of Mathematics, Fatima Jinnah Women University, The
Mall, Rawalpindi}
\email{munazzanaz@yahoo.com}
\author{Muhammad Shabir}
\address{Department of Mathematics, Quaid-i-Azam University, Islamabad}
\email{mshabirbhatti@yahoo.co.uk}
\author{Muhammad Irfan Ali}
\address{Department of Mathematics, Islamabad Model College for Boys F-7/3,
Islamabad, Pakistan.}
\email{mirfanali13@yahoo.com}

\begin{abstract}
Concept of bi-soft topological spaces is introduced. Several notions of a
soft topological space are generalized to study bi-soft topological spaces.
Separation axioms play a vital role in study of topological spaces. These
concepts have been studied in context of bi-soft topological spaces. There
is a very close relationship between topology and rough set theory. An
application of bi-soft topology is given in rough set theory.
\end{abstract}

\maketitle

\section{Introduction}

Soft set theory, initiated by Molodtsov$^{\text{\cite{Molod}}}$, is a novel
concept and a completely new approach for modeling vagueness and
uncertainty, which occur in real world problems. Applications of soft set
theory in many disciplines and real-life problems, have established their
role in scientific literature. Many researchers are working in this very
important area. Molodtsov suggests many directions for the applications of
soft sets in his seminal paper \cite{Molod}, which include smoothness of
functions, game theory, Riemann integration, Perron integration and theory
of measurement. Some important applications of soft sets are in information
systems and decision making problems can be seen $^{\text{\cite{Maji 1}\cite%
{Pie}}}$. These concepts are of utmost importance in artificial intelligence
and computer science. Algebraic structures of soft sets have been discussed
in $^{\text{\cite{Ali 1}, \cite{Ali 2}, \cite{Maji 2}}}$. Concept of Soft
topological spaces is introduced in $^{\text{\cite{Shabir}}}$, where soft
separation axioms have been studied as well. Further contributions to the
same concepts have been added by many authors in $^{\text{\cite{Bashir}, 
\cite{Cagman}, \cite{Min}}}$.

Topology is an important branch of mathematics. Separation axioms in
topology are among the most beautiful and interesting concepts. Various
generalizations of separation axioms have been studied for generalized
topological spaces. It is interesting to see that when classical notions are
replaced by new generalized concepts, several new results emerge. Kelly$^{%
\text{\cite{Kelly}}}$ introduced the concept of bitopological spaces and
studied the separation properties for bitopological spaces. These separation
axioms are actually pair-wise separation axioms. In later years, many
researchers studied bitopological spaces$^{\text{\cite{Kim}, \cite{Lal}, 
\cite{Lane}, \cite{Murdesh}, \cite{Patty}, \cite{Pervin}, \cite{Reilly}, 
\cite{Singal}}}$ due to the richness of their structure and potential for
carrying out a wide scope for the generalization of topological results in
bitopological environment. Our present work is also a continuation of this
trend.

In the present paper, concept of soft topological spaces have been
generalized to initiate the study of bi-soft topological spaces. In section %
\ref{Pre}, some preliminary concepts about bitopological spaces and soft
topological spaces are given. Section \ref{Bisofttopo} is devoted for the
study of bi-soft topological spaces. The basic structure of a bi-soft
topological space over an initial universal set $X$, with a fixed set of
parameters has been given. The concept of pair-wise soft separation axioms
for bi-soft topological spaces is studied section \ref{Sep}. Properties of
pair-wise soft $T_{0}$, $T_{1}$, and $T_{2}-$spaces and their relations with
the corresponding soft $T_{0}$, $T_{1}$, and $T_{2}-$spaces have been
discussed here. Main goal of this paper is to study the implications of
these generalized separation axioms in soft and crisp cases. Several results
in this regards have been presented. This study focuses on question: If a
pair-wise soft $T_{i}-$space $(i=0$, $1$, $2)$, say $(X,\mathcal{T}_{1},%
\mathcal{T}_{2},E)$, over a ground set $X$ is given, what can be said about
the situations,

\begin{enumerate}
\item both $(X,\mathcal{T}_{1},E)$ and $(X,\mathcal{T}_{2},E)$\ are soft $%
T_{i}-$spaces,

\item $(X,\mathcal{T}_{1}\vee \mathcal{T}_{2},E)$ is a soft $T_{i}-$space,

\item the parameterized bitopological spaces $(X,\mathcal{T}_{1e},\mathcal{T}%
_{2e})$ are $T_{i}-$spaces for all $e\in E$,

\item bi-soft subspaces $(Y,\mathcal{T}_{1Y},\mathcal{T}_{2Y},E)\ $are $%
T_{i}-$spaces for $\emptyset \neq Y\subset X$.
\end{enumerate}

Furthermore characterizations theorem is proved for pair-wise soft Hausdorff
space. Finally in section \ref{appli} an application of bi-soft topological
spaces is suggested in rough set theory.

\section{Preliminaries\label{Pre}}

In this section some basic concepts about bitopological spaces and soft
topological spaces are presented.

\begin{definition}
$^{\text{\cite{Kelly}}}$A \textit{bitopological space} is the triplet $(X,%
\mathcal{P},\mathcal{Q})$ where $X$ is a non-empty set, $\mathcal{P}$ and $%
\mathcal{Q}$ are two topologies on $X$.
\end{definition}

\begin{definition}
$^{\text{\cite{Murdesh}}}$A \textit{bitopological space} $(X,\mathcal{P},%
\mathcal{Q})$ is said to be pair-wise $T_{0}$ if for each pair of distinct
points of $X$, there is a $\mathcal{P}$-open set or a $\mathcal{Q}$-open set
containing one of the points, but not the other.
\end{definition}

\begin{definition}
$^{\text{\cite{Reilly}}}$A \textit{bitopological space} $(X,\mathcal{P},%
\mathcal{Q})$ is said to be pair-wise $T_{1}$, if for each pair of distinct
points $x$, $y$ there exist $U\in \mathcal{P}$, $V\in \mathcal{Q}$ such that 
$x\in U$, $y\notin V$ and $x\notin U$, $y\in V$.
\end{definition}

\begin{definition}
$^{\text{\cite{Kelly}}}$A \textit{bitopological space} $(X,\mathcal{P},%
\mathcal{Q})$ is said to be pair-wise $T_{2}$, if given distinct points $x$, 
$y\in X$, there exist $U\in \mathcal{P}$, $V\in \mathcal{Q}$ such that $x\in
U$, $y\in V$, $U\cap V=\emptyset $.
\end{definition}

In the following some concepts about soft sets and soft topological spaces
are given.

Let $X$ be an initial universe set and $E$ be the non-empty set of
parameters.

\begin{definition}
$^{\text{\cite{Molod}}}$ Let $U$ be an initial universe and $E$ be a set of
parameters. Let $\mathcal{P}(X)$ denotes the power set of $X$ and $A$ be a
non-empty subset of $E$. A pair $(F$,$A)$\ is called a soft set over $X$,
where $F$ is a mapping given by $F:A\rightarrow \mathcal{P}(X)$.
\end{definition}

In other words, a soft set over $X$ is a parametrized family of subsets of
the universe $X$. For $\varepsilon \in A$, $F(\varepsilon )$ may be
considered as the set of $\varepsilon -$approximate elements of the soft set 
$(F$,$A)$. Clearly, a soft set is not a set.

\begin{definition}
$^{\text{\cite{Ali 1}}}$ For two soft sets $(F$,$A)\ $and $(G$,$B)$ over a
common universe $X$, we say that $(F$,$A)$ is a \textit{soft subset} of $(G$,%
$B)$ if

\begin{enumerate}
\item $\ A\subseteq B$ and

\item $F(e)\subseteq G(e)$, for all $e\in A$.
\end{enumerate}
\end{definition}

We write $(F$,$A)\widetilde{\subset }(G$,$B)$.

$(F$,$A)$ is said to be a soft super set of $(G$,$B)$, if $(G$,$B)$ is a
soft subset of $(F$,$A)$. We denote it by $(F$,$A)\widetilde{\supset }(G$,$%
B) $.

\begin{definition}
$^{\text{\cite{Ali 1}}}$A soft set $(F$,$A)$ over $X$ is said to be a 
\textit{NULL} soft set denoted by $\Phi $\ if for all $\varepsilon \in A$, $%
F(\varepsilon )=$ $\emptyset $ $($null set$)$.
\end{definition}

\begin{definition}
$^{\text{\cite{Ali 1}}}$ A soft set $(F$,$A)$ over $X$ is said to be \textit{%
absolute} soft set denoted by $\widetilde{A}$\ if for all $e\in A$, $F(e)=$ $%
X$.
\end{definition}

\begin{definition}
$^{\text{\cite{Ali 1}}}$The \textit{Union of two soft sets} $(F$,$E)$ and $%
(G $,$E)$ over the common universe $X$\ is the soft set $(H$,$E)$, where $%
H(e)=F(e)\cup G(e)$ for all $e\in E$. We write $(F,E)\cup (G,E)=(H,E)$.
\end{definition}

\begin{definition}
$^{\text{\cite{Ali 1}}}$\textit{The intersection} of two soft sets $(F,E)$
and $(G,E)$ over a common universe $X$,\ is a soft set $(H,E)=(F,E)\cap
(G,E) $, defined by $H(e)=F(e)\cap G(e)$ for all $e\in E$.
\end{definition}

\begin{definition}
$^{\text{\cite{Shabir}}}$The difference $(H$,$E)$ of two soft sets $(F$,$E)$
and $(G$,$E)$ over $X$,\ denoted by $(F,E)-(G,E)$, is defined as $%
H(e)=F(e)-G(e)$ for all $e\in E$,
\end{definition}

\begin{definition}
$^{\text{\cite{Shabir}}}$Let $(F$,$E)$ be a soft set over $X$ and $x\in X$.
We say that $x\in (F,E)$ read as $x$ belongs to the soft set $(F$,$E)$
whenever $x\in F(\alpha )$ for all $\alpha \in E$.
\end{definition}

Note that for any $x\in X$, $x\notin (F$,$E)$, if $x\notin F(\alpha )$ for
some $\alpha \in E$.

\begin{definition}
$^{\text{\cite{Shabir}}}$Let $Y$ be a non-empty subset of $X$, then $%
\widetilde{Y}$ denotes the soft set $(Y,E)$ over $X$ for which $Y(\alpha )=Y$%
, for all $\alpha \in E$.
\end{definition}

In particular, $(X,E)$\ will be denoted by $\widetilde{X}$.

\begin{definition}
$^{\text{\cite{Shabir}}}$Let $x\in X$. Then $(x$,$E)$ denotes the soft set
over $X$ for which $x(\alpha )=\{x\}$, for all $\alpha \in E$.
\end{definition}

\begin{definition}
$^{\text{\cite{Shabir}}}$Let $(F$,$E)$ be a soft set over $X$ and $Y$\ be a
non-empty subset of $X$. Then the soft subset of $(F$,$E)$ over $Y$ denoted
by $(^{Y}F$,$E)$, is defined as follows%
\begin{equation*}
^{Y}F(\alpha )=Y\cap F(\alpha ),\text{ for all }\alpha \in E
\end{equation*}%
In other words $(^{Y}F,E)=\widetilde{Y}\cap (F,E)$.
\end{definition}

\begin{definition}
$^{\text{\cite{Shabir}}}$The\textit{\ complement of a soft set} $(F,E)$ is
denoted by $(F,E)^{c}$ and is defined by $(F,E)^{c}=(F^{c},E)$ where $%
F^{c}:E\rightarrow \mathcal{P}(X)$ is a mapping given by%
\begin{equation*}
F^{c}(\alpha )=X-F(\alpha )\text{ for all}\ \alpha \in E.
\end{equation*}
\end{definition}

\begin{proposition}
$^{\text{\cite{Shabir}}}$Let $(F,E)$ and$\ (G,E)$ be the soft sets over $X$.
Then

\begin{enumerate}
\item $((F,E)\cup (G,E))^{c}=(F,E)^{c}\cap (G,E)^{c}$,

\item $((F,E)\cap (G,E)^{c}=(F,E)^{c}\cup (G,E)^{c}$.
\end{enumerate}
\end{proposition}

\begin{definition}
$^{\text{\cite{Shabir}}}$Let $\mathcal{T}$ be the collection of soft sets
over $X$. Then $\mathcal{T}$ is said to be a soft topology on $X$ if

\begin{enumerate}
\item $\Phi $, $\widetilde{X}$ belong to $\mathcal{T}$

\item the union of any number of soft sets in $\mathcal{T}$\ belongs to $%
\mathcal{T}$

\item the intersection of any two soft sets in $\mathcal{T}$\ belongs to\ $%
\mathcal{T}$.
\end{enumerate}

The triplet $(X$,$\mathcal{T}$,$E)$ is called a soft topological space over $%
X$.
\end{definition}

\begin{example}
$^{\text{\cite{Shabir}}}$Let $X=\{h_{1}$,$h_{2}$,$h_{3}\}$, $E=\{e_{1}$,$%
e_{2}\}$ and

$\mathcal{T}=\{\Phi $,$\widetilde{X}$,$(F_{1}$,$E)$,$(F_{2}$,$E)$,$(F_{3}$,$%
E)$,$(F_{4}$,$E)$, $(F_{5}$,$E)\}$ where $(F_{1}$,$E)$,$(F_{2}$,$E)$,$(F_{3}$%
,$E)$,$(F_{4}$,$E)$, and $(F_{5}$,$E)$ are soft sets over $X$, defined as
follows:%
\begin{equation*}
\begin{array}[t]{ll}
F_{1}(e_{1})=\{h_{2}\}, & F_{1}(e_{2})=\{h_{1}\}, \\ 
F_{2}(e_{1})=\{h_{2},h_{3}\}, & F_{2}(e_{2})=\{h_{1},h_{2}\}, \\ 
F_{3}(e_{1})=\{h_{1},h_{2}\}, & F_{3}(e_{2})=X, \\ 
F_{4}(e_{1})=\{h_{1},h_{2}\}, & F_{4}(e_{2})=\{h_{1},h_{3}\}, \\ 
F_{5}(e_{1})=\{h_{2}\}, & F_{4}(e_{2})=\{h_{1},h_{2}\}.%
\end{array}%
\end{equation*}%
Then $\mathcal{T}$\ defines a soft topology on $X$ and hence $(X$,$\mathcal{T%
}$,$E)$ is a soft topological space over $X$.
\end{example}

\begin{definition}
$^{\text{\cite{Shabir}}}$Let $(X$,$\mathcal{T}$,$E)$ be a soft topological
space over $X$. Then the members of$\ \mathcal{T}$ are said to be soft open
sets in $X$.
\end{definition}

\begin{definition}
$^{\text{\cite{Shabir}}}$Let $(X$,$\mathcal{T}$,$E)$ be a soft topological
space over $X$. A soft set $(F,E)$ over $X$ is said to be a soft closed set
in $X$, if its relative complement $(F,E)^{c}$ belongs to $\mathcal{T}$.
\end{definition}

\begin{definition}
\label{softcl}$^{\text{\cite{Shabir}}}$Let $(X$,$\mathcal{T}$,$E)$ be a 
\textit{soft topological space }over $X$ and $(F$,$E)$ be a soft set over $X$%
. Then the \textit{soft closure} of $(F$,$E)$, denoted by $\overline{(F,E)}$
is the intersection of all soft closed super sets of $(F$,$E)$.
\end{definition}

\begin{definition}
$^{\text{\cite{Shabir}}}$Let $(X$,$\mathcal{T}$,$E)$ be a \textit{soft
topological space} over $X$ and $Y$\ be a non-empty subset of $X$. Then%
\begin{equation*}
\mathcal{T}_{Y}=\{\text{ }(^{Y}F,E)\text{\ }|\text{ }(F,E)\in \mathcal{T}%
\text{\ }\}
\end{equation*}

is said to be the \textit{soft relative topology} on $Y$ and $(Y$,$\mathcal{T%
}_{Y}$,$E)$\ is called a \textit{soft subspace} of $(X$,$\mathcal{T}$,$E)$.
\end{definition}

\begin{definition}
$^{\text{\cite{Shabir}}}$Let $(X$,$\mathcal{T}$,$E)$ be a soft topological
space over $X$ and $x$,$y\in X$ be such that $x\neq y$. If there exist soft
open sets $(F$,$E)$ and $(G$,$E)$ such that

"$x\in (F$,$E)$ and$\ y\notin (F$,$E)$" or "$y\in (G$,$E)$ and $x\notin (G$,$%
E)$", then $(X$,$\mathcal{T}$,$E)$\ is called a soft $T_{0}-$space.
\end{definition}

\begin{definition}
$^{\text{\cite{Shabir}}}$Let $(X$,$\mathcal{T}$,$E)$ be a soft topological
space over $X$ and $x$,$y\in X$ be such that $x\neq y$. If there exist soft
open sets $(F$,$E)$ and $(G$,$E)$ such that

"$x\in (F$,$E)$ and$\ y\notin (F$,$E)$" and "$y\in (G$,$E)$ and $x\notin (G$,%
$E)$", then $(X$,$\mathcal{T}$,$E)$\ is called a soft $T_{1}-$space.
\end{definition}

\begin{definition}
$^{\text{\cite{Shabir}}}$Let $(X$,$\mathcal{T}$,$E)$ be a soft topological
space over $X$ and $x$,$y\in X$ be such that $x\neq y$. If there exist soft
open sets $(F$,$E)$ and $(G$,$E)$ such that

$x\in (F$,$E)$, $y\in (G$,$E)$ and $(F$,$E)\cap (G$,$E)=\Phi $, then $(X$,$%
\mathcal{T}$,$E)$\ is called a soft $T_{2}-$space.
\end{definition}

\section{Bi-Soft Topological Spaces\label{Bisofttopo}}

In this section study of bi-soft topological spaces is initiated.

\begin{definition}
Let $\mathcal{T}_{1}$ and $\mathcal{T}_{2}$ be two soft topologies on $X$.
Then the quadruple $(X,\mathcal{T}_{1},\mathcal{T}_{2},E)$ is said to be a
bi-soft topological space over $X$.
\end{definition}

\begin{example}
Let $X=\{h_{1},h_{2},h_{3}\}$, $E=\{e_{1},e_{2}\}$. Let%
\begin{eqnarray*}
\mathcal{T}_{1} &=&\{\Phi ,\widetilde{X},(F_{1},E),(F_{2},E)\}\text{, and} \\
\mathcal{T}_{2} &=&\{\Phi ,\widetilde{X}%
,(G_{1},E),(G_{2},E),(G_{3},E),(G_{4},E)\}\text{,}
\end{eqnarray*}%
where $(F_{1},E),(F_{2},E),(G_{1},E),(G_{2},E),(G_{3},E),(G_{4},E)$ are soft
sets over $X$, defined as follows:%
\begin{equation*}
\begin{array}{ll}
F_{1}(e_{1})=\{h_{1}\}, & F_{1}(e_{2})=\{h_{1},h_{2}\}, \\ 
F_{2}(e_{1})=\{h_{1},h_{3}\}, & F_{2}(e_{2})=X,%
\end{array}%
\end{equation*}%
and 
\begin{equation*}
\begin{array}{ll}
G_{1}(e_{1})=\{h_{1}\}, & G_{1}(e_{2})=\{h_{2}\}, \\ 
G_{2}(e_{1})=\{h_{1},h_{2}\}, & G_{2}(e_{2})=\{h_{2}\}, \\ 
G_{3}(e_{1})=\{h_{2}\}, & G_{3}(e_{2})=\{h_{2}\}, \\ 
G_{4}(e_{1})=\{\}, & G_{4}(e_{2})=\{h_{2}\}.%
\end{array}%
\end{equation*}%
Then $\mathcal{T}_{1}$\ and $\mathcal{T}_{2}$ are soft topologies on $X$.
Thus $(X,\mathcal{T}_{1},\mathcal{T}_{2},E)$ is a \textit{bi-soft
topological space over }$X$.
\end{example}

\begin{proposition}
\label{Propref}Let $(X,\mathcal{T}_{1},\mathcal{T}_{2},E)$ be a bi-soft
topological space over $X$. We define%
\begin{eqnarray*}
\mathcal{T}_{1e} &=&\{F(e)|(F,E)\in \mathcal{T}_{1}\} \\
\mathcal{T}_{2e} &=&\{G(e)|(G,E)\in \mathcal{T}_{2}\}
\end{eqnarray*}%
for each $e\in E$. Then $(X,\mathcal{T}_{1e},\mathcal{T}_{2e})$ is a \textit{%
bitopological space}.
\end{proposition}

\begin{proof}
Follows from the fact that $\mathcal{T}_{1e}$, and $\mathcal{T}_{2e}$ are
topologies on $X$ for each $e\in E$.
\end{proof}

Proposition \ref{Propref}\ shows that corresponding to each parameter $e\in
E $, we have a bitopological space $X$. Thus a bi-soft topology on $X$ gives
a parameterized family of bitopological spaces.

\begin{example}
Let $X=\{h_{1},h_{2},h_{3}\}$, $E=\{e_{1},e_{2}\}$ and%
\begin{eqnarray*}
\mathcal{T}_{1} &=&\{\Phi ,\widetilde{X}%
,(F_{1},E),(F_{2},E),(F_{3},E),(F_{4},E),(F_{5},E)\}\text{, and} \\
\mathcal{T}_{2} &=&\{\Phi ,\widetilde{X}%
,(G_{1},E),(G_{2},E),(G_{3},E),(G_{4},E)\}\text{,}
\end{eqnarray*}%
where $%
(F_{1},E),(F_{2},E),(F_{3},E),(F_{4},E),(F_{5},E),(G_{1},E),(G_{2},E),(G_{3},E), 
$ and $(G_{4},E)$ are soft sets over $X$, defined as follows:%
\begin{equation*}
\begin{array}{ll}
F_{1}(e_{1})=\{h_{2}\}, & F_{1}(e_{2})=\{h_{1}\}, \\ 
F_{2}(e_{1})=\{h_{2},h_{3}\}, & F_{2}(e_{2})=\{h_{1},h_{2}\}, \\ 
F_{3}(e_{1})=\{h_{1},h_{2}\}, & F_{3}(e_{2})=X, \\ 
F_{4}(e_{1})=\{h_{1},h_{2}\}, & F_{4}(e_{2})=\{h_{1},h_{3}\}, \\ 
F_{5}(e_{1})=\{h_{2}\}, & F_{5}(e_{2})=\{h_{1},h_{2}\},%
\end{array}%
\end{equation*}%
and%
\begin{equation*}
\begin{array}{ll}
G_{1}(e_{1})=\{h_{1}\}, & G_{1}(e_{2})=\{h_{2}\}, \\ 
G_{2}(e_{1})=\{h_{1},h_{2}\}, & G_{2}(e_{2})=\{h_{2}\}, \\ 
G_{3}(e_{1})=\{h_{2}\}, & G_{3}(e_{2})=\{h_{2}\}, \\ 
G_{4}(e_{1})=\{\}, & G_{4}(e_{2})=\{h_{2}\}.%
\end{array}%
\end{equation*}%
Then $\mathcal{T}_{1}$\ and $\mathcal{T}_{2}$ are soft topologies on $X$.
Therefore $(X,\mathcal{T}_{1},\mathcal{T}_{2},E)$ is a \textit{bi-soft
topological space over }$X$. It can be easily seen that%
\begin{eqnarray*}
\mathcal{T}_{1e_{1}} &=&\{\emptyset
,X,\{h_{2}\},\{h_{1},h_{2}\},\{h_{2},h_{3}\}\}, \\
\mathcal{T}_{2e_{1}} &=&\{\emptyset ,X,\{h_{1}\},\{h_{2}\},\{h_{1},h_{2}\}\},
\end{eqnarray*}%
and%
\begin{eqnarray*}
\mathcal{T}_{1e_{2}} &=&\{\emptyset
,X,\{h_{1}\},\{h_{1},h_{3}\},\{h_{1},h_{2}\}\}, \\
\mathcal{T}_{2e_{2}} &=&\{\emptyset ,X,\{h_{2}\}\},
\end{eqnarray*}%
are topologies on $X$. Thus $(X,\mathcal{T}_{1e_{1}},\mathcal{T}_{2e_{1}})$
and $(X,\mathcal{T}_{1e_{2}},\mathcal{T}_{2e_{2}})$ are bitopological spaces
corresponding to parameters.
\end{example}

We have seen in \cite{Shabir} that the intersection of two soft topologies
is again a soft topology on $X$ but the union of two soft topologies need
not be a soft topology and its examples can be found in \cite{Shabir}. Now
we define the supremum soft topology:

\begin{definition}
Let $(X,\mathcal{T}_{1},E)$ and $(X,\mathcal{T}_{2},E)$\ be two soft
topological spaces over $X.$ Let $\mathcal{T}_{1}\vee \mathcal{T}_{2}$ be
the smallest soft topology on $X$ that contains $\mathcal{T}_{1}\cup 
\mathcal{T}_{2}$.
\end{definition}

\begin{example}
Let $X=\{h_{1},h_{2},h_{3}\}$, $E=\{e_{1},e_{2}\}$. Let%
\begin{eqnarray*}
\mathcal{T}_{1} &=&\{\Phi ,\widetilde{X},(F_{1},E),(F_{2},E)\}\text{, and} \\
\mathcal{T}_{2} &=&\{\Phi ,\widetilde{X}%
,(G_{1},E),(G_{2},E),(G_{3},E),(G_{4},E)\}\text{,}
\end{eqnarray*}%
where $(F_{1},E),(F_{2},E),(G_{1},E),(G_{2},E),(G_{3},E),(G_{4},E)$ are soft
sets over $X$, defined as follows:%
\begin{equation*}
\begin{array}{lll}
F_{1}(e_{1})=\{h_{1}\}, &  & F_{1}(e_{2})=\{h_{1},h_{2}\}, \\ 
F_{2}(e_{1})=\{h_{1},h_{3}\}, &  & F_{2}(e_{2})=X,%
\end{array}%
\end{equation*}%
and 
\begin{equation*}
\begin{array}{lll}
G_{1}(e_{1})=\{h_{1}\}, &  & G_{1}(e_{2})=\{h_{2}\}, \\ 
G_{2}(e_{1})=\{h_{1},h_{2}\}, &  & G_{2}(e_{2})=\{h_{2}\}, \\ 
G_{3}(e_{1})=\{h_{2}\}, &  & G_{3}(e_{2})=\{h_{2}\}, \\ 
G_{4}(e_{1})=\{\}, &  & G_{4}(e_{2})=\{h_{2}\}.%
\end{array}%
\end{equation*}%
Then $\mathcal{T}_{1}$\ and $\mathcal{T}_{2}$ are soft topologies on $X$. Now%
\begin{equation*}
\mathcal{T}_{1}\vee \mathcal{T}_{2}=\{\Phi ,\widetilde{X}%
,(F_{1},E),(F_{2},E),(G_{1},E),(G_{2},E),(G_{3},E),(G_{4},E),(H_{1},E)\}
\end{equation*}%
where%
\begin{equation*}
\begin{array}{lll}
H_{1}(e_{1})=\{h_{1},h_{2}\}, &  & H_{1}(e_{2})=\{h_{1},h_{2}\},%
\end{array}%
\end{equation*}%
Thus $(X,\mathcal{T}_{1}\vee \mathcal{T}_{2},E)$ is the smallest \textit{%
soft topological space over }$X$ that contains $\mathcal{T}_{1}\cup \mathcal{%
T}_{2}$.
\end{example}

\section{Bi-Soft Separation Axioms\label{Sep}}

In the last section concept of bi-soft topological spaces has been
introduced. In this section separation axioms for bi-soft topological spaces
are being studied.

\begin{definition}
A bi-soft topological space $(X,\mathcal{T}_{1},\mathcal{T}_{2},E)$ over $X$
is said to be \textit{pair-wise soft }$T_{0}-$\textit{space} if for every
pair of distinct points $x$,$y\in X$, there is a $\mathcal{T}_{1}-$soft open
set $(F,E)$ such that $x\in (F,E)$ and $y\notin (F,E)$ or a $\mathcal{T}%
_{2}- $soft open set $(G,E)$ such that $x\notin (G,E)$ and $y\in (G,E)$.
\end{definition}

\begin{example}
\label{Examp2}Let $X$ be any non-empty set and $E$ be a set of parameters.
Consider%
\begin{eqnarray*}
\mathcal{T}_{1} &=&\{\Phi ,\widetilde{X}\}\text{ \ Soft indiscrete topology
over }X \\
\mathcal{T}_{2} &=&\{(F,E)|(F,E)\text{ is a soft set over }X\}\text{ \ \
Soft discrete topology over }X
\end{eqnarray*}%
Then $(X,\mathcal{T}_{1},\mathcal{T}_{2},E)$\ is a pair-wise soft $T_{0}-$%
space.
\end{example}

\begin{proposition}
\label{Prop2}Let $(X,\mathcal{T}_{1},\mathcal{T}_{2},E)$ be a bi-soft
topological space over $X$. If $(X,\mathcal{T}_{1},E)$\ or $(X,\mathcal{T}%
_{2},E)$\ is a soft $T_{0}-$space then $(X,\mathcal{T}_{1},\mathcal{T}%
_{2},E) $\ is a pair-wise soft $T_{0}-$space.

\begin{proof}
Let $x$,$y\in X$ be such that $x\neq y$. Suppose that $(X,\mathcal{T}_{2},E)$%
\ is a soft $T_{0}-$space. Then there exist some $(F,E)\in \mathcal{T}_{1}$
such that $x\in (F,E)$ and$\ y\notin (F,E)$ or some $(G,E)\in \mathcal{T}%
_{2} $ such that $y\in (G,E)$ and $x\notin (G,E)$. In either case we obtain
the requirement and so $(X,\mathcal{T}_{1},\mathcal{T}_{2},E)$\ is a
pair-wise soft $T_{0}-$space.
\end{proof}
\end{proposition}

\begin{remark}
The converse of Proposition \ref{Prop2} is not true in general.
\end{remark}

\begin{example}
Let $X=\{h_{1},h_{2},h_{3},h_{4}\}$, $E=\{e_{1},e_{2}\}$ and%
\begin{eqnarray*}
\mathcal{T}_{1} &=&\{\Phi ,\widetilde{X},(F,E)\}\text{, and} \\
\mathcal{T}_{2} &=&\{\Phi ,\widetilde{X},(G_{1},E),(G_{2},E),(G_{3},E)\}%
\text{,}
\end{eqnarray*}%
where $(F,E),(G_{1},E),(G_{2},E),$ and $(G_{3},E)$ are soft sets over $X$,
defined as follows:%
\begin{equation*}
\begin{array}{lll}
F(e_{1})=\{h_{1},h_{3}\}, &  & F(e_{2})=\{h_{3}\},%
\end{array}%
\end{equation*}%
and%
\begin{equation*}
\begin{array}{ll}
G_{1}(e_{1})=\{h_{3},h_{4}\}, & G_{1}(e_{2})=\{h_{1},h_{4}\}, \\ 
G_{2}(e_{1})=\{h_{2}\}, & G_{2}(e_{2})=\{h_{2}\}, \\ 
G_{3}(e_{1})=\{h_{2},h_{3},h_{4}\}, & G_{3}(e_{2})=\{h_{1},h_{2},h_{4}\}%
\text{.}%
\end{array}%
\end{equation*}%
Then $\mathcal{T}_{1}$\ and $\mathcal{T}_{2}$ are soft topologies on $X$.
Therefore $(X,\mathcal{T}_{1},\mathcal{T}_{2},E)$ is a \textit{bi-soft
topological space over }$X$.

Now $h_{1},h_{2}\in X$ and $(G_{2},E)\in \mathcal{T}_{2}$ such that 
\begin{equation*}
h_{2}\in (G_{2},E)\text{, }h_{1}\notin (G_{2},E)\text{.}
\end{equation*}%
$h_{1},h_{3}\in X$ and $(F,E)\in \mathcal{T}_{1}$ such that 
\begin{equation*}
h_{3}\in (F,E)\text{, }h_{1}\notin (F,E)\text{.}
\end{equation*}%
$h_{1},h_{4}\in X$ and $(G_{1},E)\in \mathcal{T}_{2}$ such that 
\begin{equation*}
h_{4}\in (G_{1},E)\text{, }h_{1}\notin (G_{1},E)\text{.}
\end{equation*}%
$h_{2},h_{3}\in X$ and $(G_{2},E)\in \mathcal{T}_{2}$ such that 
\begin{equation*}
h_{2}\in (G_{2},E)\text{, }h_{3}\notin (G_{2},E)\text{.}
\end{equation*}%
$h_{2},h_{4}\in X$ and $(G_{2},E)\in \mathcal{T}_{2}$ such that 
\begin{equation*}
h_{2}\in (G_{2},E)\text{, }h_{4}\notin (G_{2},E)\text{.}
\end{equation*}%
Finally $h_{3},h_{4}\in X$ and $(G_{3},E)\in \mathcal{T}_{2}$ such that 
\begin{equation*}
h_{4}\in (G_{3},E)\text{, }h_{3}\notin (G_{3},E)\text{.}
\end{equation*}%
Thus $(X,\mathcal{T}_{1},\mathcal{T}_{2},E)$ is a pair-wise \textit{soft }$%
T_{0}-$\textit{\ space over }$X$.

We observe that $h_{1},h_{2}\in X$ and there does not exist any $(F,E)\in 
\mathcal{T}_{1}$ such that $h_{1}\in (F,E)$, $h_{2}\notin (F,E)$ or $%
h_{2}\in (F,E)$, $h_{1}\notin (F,E)$. Therefore $(X,\mathcal{T}_{1},E)$ is
not a \textit{soft }$T_{0}-$\textit{\ space over }$X$. Similarly $%
h_{1},h_{3}\in X$ and there does not exist any $(G,E)\in \mathcal{T}_{2}$
such that $h_{1}\in (G,E)$, $h_{3}\notin (G,E)$ or $h_{3}\in (G,E)$, $%
h_{1}\notin (G,E)$ so $(X,\mathcal{T}_{2},E)$ is not a \textit{soft }$T_{0}-$%
\textit{\ space also}.
\end{example}

\begin{proposition}
\label{Prop1}Let $(X,\mathcal{T}_{1},\mathcal{T}_{2},E)$ be a bi-soft
topological space over $X$. If $(X,\mathcal{T}_{1},\mathcal{T}_{2},E)$\ is a
pair-wise soft $T_{0}-$space then $(X,\mathcal{T}_{1}\vee \mathcal{T}_{2},E)$%
\ is a soft $T_{0}-$space.

\begin{proof}
Let $x$,$y\in X$ be such that $x\neq y$. Then there exists some $(F,E)\in 
\mathcal{T}_{1}$ such that $x\in (F,E)$ and$\ y\notin (F,E)$ or some $%
(G,E)\in \mathcal{T}_{2}$ such that $y\in (G,E)$ and $x\notin (G,E)$. In
either case $(F,E),$ $(G,E)\in \mathcal{T}_{1}\vee \mathcal{T}_{2}$.\ Hence $%
(X,\mathcal{T}_{1}\vee \mathcal{T}_{2},E)$\ is a soft $T_{0}-$space.
\end{proof}
\end{proposition}

\begin{remark}
The converse of Proposition \ref{Prop1}, is not true. This is shown by the
following example:
\end{remark}

\begin{example}
Let $X=\{h_{1},h_{2},h_{3},h_{4}\}$, $E=\{e_{1},e_{2}\}$ and%
\begin{eqnarray*}
\mathcal{T}_{1} &=&\{\Phi ,\widetilde{X},(F_{1},E),(F_{2},E)\}\text{, and} \\
\mathcal{T}_{2} &=&\{\Phi ,\widetilde{X},(G,E)\}\text{,}
\end{eqnarray*}%
where $(F_{1},E),(F_{2},E),$ and $(G,E)$ are soft sets over $X$, defined as
follows:%
\begin{equation*}
\begin{array}{ll}
F_{1}(e_{1})=\{h_{1},h_{4}\}, & F_{1}(e_{2})=\{h_{4}\}, \\ 
F_{2}(e_{1})=\{h_{4}\}, & F_{2}(e_{2})=\{\},%
\end{array}%
\end{equation*}%
and%
\begin{equation*}
\begin{array}{lll}
G(e_{1})=\{h_{2},h_{4}\}, &  & G(e_{2})=\{h_{1},h_{2}\}.%
\end{array}%
\end{equation*}%
Then $\mathcal{T}_{1}$\ and $\mathcal{T}_{2}$ are soft topologies on $X$.
Therefore $(X,\mathcal{T}_{1},\mathcal{T}_{2},E)$ is a \textit{bi-soft
topological space over }$X$. Now%
\begin{equation*}
\mathcal{T}_{1}\vee \mathcal{T}_{2}=\{\Phi ,\widetilde{X}%
,(F_{1},E),(F_{2},E),(G,E),(H,E)\}
\end{equation*}%
where%
\begin{equation*}
\begin{array}{lll}
H(e_{1})=\{h_{1},h_{2},h_{4}\}, &  & H(e_{2})=\{h_{1},h_{2},h_{4}\},%
\end{array}%
\end{equation*}%
so $(X,\mathcal{T}_{1}\vee \mathcal{T}_{2},E)$ is a \textit{soft topological
space over }$X$ that contains $\mathcal{T}_{1}\cup \mathcal{T}_{2}$.

For $h_{1},h_{3}\in X$, we cannot find any soft sets $(F,E)\in \mathcal{T}%
_{1}$ or $(G,E)\in \mathcal{T}_{2}$ such that 
\begin{eqnarray*}
h_{1} &\in &(F,E)\text{, }h_{3}\notin (F,E)\text{ or} \\
h_{3} &\in &(G,E)\text{, }h_{1}\notin (G,E)\text{.}
\end{eqnarray*}%
Thus $(X,\mathcal{T}_{1},\mathcal{T}_{2},E)$ is not pair-wise \textit{soft} $%
T_{0}-$space.

Now $h_{1},h_{2}\in X$ and $(G,E)\in \mathcal{T}_{2}$ such that 
\begin{equation*}
h_{2}\in (G,E)\text{, }h_{1}\notin (G,E)\text{.}
\end{equation*}%
$h_{1},h_{3}\in X$ and $(H,E)\in \mathcal{T}_{1}\vee \mathcal{T}_{2}$ such
that 
\begin{equation*}
h_{1}\in (H,E)\text{, }h_{3}\notin (H,E)\text{.}
\end{equation*}%
$h_{1},h_{4}\in X$ and $(F_{1},E)\in \mathcal{T}_{1}$ such that 
\begin{equation*}
h_{4}\in (F_{1},E)\text{, }h_{1}\notin (F_{1},E)\text{.}
\end{equation*}%
$h_{2},h_{3}\in X$ and $(G,E)\in \mathcal{T}_{2}$ such that 
\begin{equation*}
h_{2}\in (G,E)\text{, }h_{3}\notin (G,E)\text{.}
\end{equation*}%
$h_{2},h_{4}\in X$ and $(F_{1},E)\in \mathcal{T}_{1}$ such that 
\begin{equation*}
h_{4}\in (F_{1},E)\text{, }h_{2}\notin (F_{1},E)\text{.}
\end{equation*}%
Finally $h_{3},h_{4}\in X$ and $(F_{1},E)\in \mathcal{T}_{1}$ such that 
\begin{equation*}
h_{4}\in (F_{1},E)\text{, }h_{3}\notin (F_{1},E)\text{.}
\end{equation*}%
Thus $(X,\mathcal{T}_{1}\vee \mathcal{T}_{2},E)$ is a \textit{soft }$T_{0}-$%
\textit{\ space over }$X$.
\end{example}

\begin{example}
\label{Examp1}Let $X=\{h_{1},h_{2},h_{3},h_{4}\}$, $E=\{e_{1},e_{2}\}$ and%
\begin{eqnarray*}
\mathcal{T}_{1} &=&\{\Phi ,\widetilde{X},(F,E)\}\text{, and} \\
\mathcal{T}_{2} &=&\{\Phi ,\widetilde{X},(G_{1},E),(G_{2},E),(G_{3},E)\}%
\text{,}
\end{eqnarray*}%
where $(F,E),(G_{1},E),(G_{2},E),$ and $(G_{3},E)$ are soft sets over $X$,
defined as follows:%
\begin{equation*}
\begin{array}{lll}
F(e_{1})=\{h_{1},h_{3}\}, &  & F(e_{2})=\{h_{3}\},%
\end{array}%
\end{equation*}%
and%
\begin{equation*}
\begin{array}{lll}
G_{1}(e_{1})=\{h_{3},h_{4}\}, &  & G_{1}(e_{2})=\{h_{1},h_{4}\}, \\ 
G_{2}(e_{1})=\{h_{2}\}, &  & G_{2}(e_{2})=\{h_{2}\}, \\ 
G_{3}(e_{1})=\{h_{2},h_{3},h_{4}\}, &  & G_{3}(e_{2})=\{h_{1},h_{2},h_{4}\}%
\text{.}%
\end{array}%
\end{equation*}%
Then $\mathcal{T}_{1}$\ and $\mathcal{T}_{2}$ are soft topologies on $X$.
Therefore $(X,\mathcal{T}_{1},\mathcal{T}_{2},E)$ is a \textit{bi-soft
topological space over }$X$. Also observe that $(X,\mathcal{T}_{1},\mathcal{T%
}_{2},E)$ is a pair-wise soft $T_{0}-$space. Now%
\begin{eqnarray*}
\mathcal{T}_{1e_{1}} &=&\{\emptyset ,X,\{h_{1},h_{3}\}\}, \\
\mathcal{T}_{2e_{1}} &=&\{\emptyset ,X,\{h_{3}\}\},
\end{eqnarray*}%
and%
\begin{eqnarray*}
\mathcal{T}_{2e_{2}} &=&\{\emptyset
,X,\{h_{2}\},\{h_{3},h_{4}\},\{h_{2},h_{3},h_{4}\}\}, \\
\mathcal{T}_{2e_{2}} &=&\{\emptyset
,X,\{h_{2}\},\{h_{1},h_{4}\},\{h_{1},h_{2},h_{4}\}\},
\end{eqnarray*}%
are corresponding parametrized topologies on $X$. Considering the \textit{%
bitopological space} $(X,\mathcal{T}_{1e_{1}},\mathcal{T}_{2e_{1}})$, one
can easily see that $h_{2},h_{4}\in X$ and there do not exist any $\mathcal{T%
}_{1e_{1}}-$open set $X$ such that $h_{2}\in X$, $h_{4}\notin X$ or $%
\mathcal{T}_{2e_{1}}-$open set $V$ such that $h_{4}\in V$, $h_{2}\notin V$.
Thus $(X,\mathcal{T}_{1e_{1}},\mathcal{T}_{2e_{1}})$ is not a pair-wise $%
T_{0}-$space.
\end{example}

Example \ref{Examp1}, shows that the parametrized bitopological spaces need
not be pair-wise $T_{0}$ even if the given bi-soft topological space is
pair-wise soft $T_{0}-$space. Following proposition will provide us an
alternative condition that resolves this problem while looking for the
corresponding parameterized families.

\begin{proposition}
Let $(X,\mathcal{T}_{1},\mathcal{T}_{2},E)$ be a bi-soft topological space
over $X$ and $x$,$y\in X$ be such that $x\neq y$. If there exists a $%
\mathcal{T}_{1}-$soft open set $(F,E)$ such that $x\in (F$,$E)$ and$\ y\in
(F $,$E)^{c}$ or a $\mathcal{T}_{2}-$soft open set $(G,E)$ such that $y\in
(G $,$E)$ and $x\in (G$,$E)^{c}$, then $(X,\mathcal{T}_{1},\mathcal{T}%
_{2},E) $ is a pair-wise soft $T_{0}-$space over $X$ and $(X,\mathcal{T}%
_{1e},\mathcal{T}_{2e})$\ is a pair-wise $T_{0}-$space for each $e\in E$.
\end{proposition}

\begin{proof}
Let $x$,$y\in X$ be such that $x\neq y$ and $(F,E)\in \mathcal{T}_{1}$ such
that $x\in (F,E)$ and$\ y\in (F,E)^{c}$ Or $(G,E)\in \mathcal{T}_{2}$ such
that $y\in (G,E)$ and $x\in (G,E)^{c}$. If $y\in (F,E)^{c}$\ then $y\in
(F(e))^{c}$\ for each $e\in E$. This implies that $y\notin F(e)$\ for each $%
e\in E$. Therefore $y\notin (F,E)$.\ Similarly we can show that if $x\in
(G,E)^{c}$ then $x\notin (G,E)$.\ Hence $(X,\mathcal{T}_{1},\mathcal{T}%
_{2},E)$\ is a pair-wise soft $T_{0}-$space.\ Now for any $e\in E$,\ $(X,%
\mathcal{T}_{1e},\mathcal{T}_{2e})$\ is a \textit{bitopological space}. By
above discussion we have $x\in F(e)\in \mathcal{T}_{1e}$ and\ $y\notin F(e)$
or\ $y\in G(e)\in \mathcal{T}_{2e}$ and\ $x\notin G(e)$. Thus $(X,\mathcal{T}%
_{1e},\mathcal{T}_{2e})$ is a pair-wise $T_{0}-$space.
\end{proof}

\begin{proposition}
Let $(X,\mathcal{T}_{1},\mathcal{T}_{2},E)$ be a bi-soft topological space
over $X$ and $Y$ be a non-empty subset of $X$. If $(X,\mathcal{T}_{1},%
\mathcal{T}_{2},E)$\ is a pair-wise soft $T_{0}-$space then $(Y,\mathcal{T}%
_{1Y},\mathcal{T}_{2Y},E)$\ is also a pair-wise soft $T_{0}-$space.

\begin{proof}
Let $x$,$y\in Y$ be such that $x\neq y$. Then there exists some soft set $%
(F,E)\in \mathcal{T}_{1}$ or $(G,E)\in \mathcal{T}_{2}$ such that $x\in
(F,E) $ and$\ y\notin (F,E)$ or $y\in (G,E)$ and $x\notin (G,E)$. Suppose
that there exists some soft set $(F,E)\in \mathcal{T}_{1}$ such that $x\in
(F,E)$ and$\ y\notin (F,E)$. Now $x\in Y$\ implies that $x\in \widetilde{Y}$%
.\ So $x\in \widetilde{Y}$ and $x\in (F,E)$. Hence $x\in \widetilde{Y}\cap
(F,E)=(^{Y}F,E)$. Consider $y\notin (F,E)$, this means that $y\notin F(e)$
for some $e\in E$.\ Then $y\notin Y\cap F(e)=Y(e)\cap F(e)$.\ Therefore $%
y\notin \widetilde{Y}\cap (F,E)=(^{Y}F,E)$. Similarly it can be proved that
if $y\in (G,E)$ and $x\notin (G,E)$ then $y\in (^{Y}G,E)$ and $x\notin
(^{Y}G,E)$.\ Thus $(Y,\mathcal{T}_{1Y},\mathcal{T}_{2Y},E)$\ is also a
pair-wise soft $T_{0}-$space.
\end{proof}
\end{proposition}

\begin{definition}
A bi-soft topological space $(X,\mathcal{T}_{1},\mathcal{T}_{2},E)$ over $X$
is said to be pair-wise soft $T_{1}-$space if for every pair of distinct
points $x$,$y\in X$, there is a $\mathcal{T}_{1}-$soft open set $(F,E)$ such
that $x\in (F,E)$ and $y\notin (F,E)$ and a $\mathcal{T}_{2}-$soft open set $%
(G,E)$ such that $x\notin (G,E)$ and $y\in (G,E)$.
\end{definition}

\begin{example}
\label{Examp3}Let $X=\{h_{1},h_{2},h_{3}\}$, $E=\{e_{1},e_{2}\}$ and%
\begin{eqnarray*}
\mathcal{T}_{1} &=&\{\Phi ,\widetilde{X}%
,(F_{1},E),(F_{2},E),(F_{3},E),(F_{4},E),(F_{5},E),(F_{6},E),(F_{7},E)\}%
\text{, and} \\
\mathcal{T}_{2} &=&\{\Phi ,\widetilde{X}%
,(G_{1},E),(G_{2},E),(G_{3},E),(G_{4},E),(G_{5},E),(G_{6},E)\}\text{,}
\end{eqnarray*}%
where $(F_{1},E),$ $(F_{2},E),$ $(F_{3},E),$ $(F_{4},E),$ $(F_{5},E),$ $%
(F_{6},E),$ $(F_{7},E),$ $(F_{8},E),$ $(G_{1},E),$ $(G_{2},E),$ $(G_{3},E),$ 
$(G_{4},E),$ $(G_{5},E)$ and $(G_{6},E)$ are soft sets over $X$, defined as
follows:%
\begin{equation*}
\begin{array}{ll}
F_{1}(e_{1})=\{h_{1}\}, & F_{1}(e_{2})=\{h_{1},h_{3}\}, \\ 
F_{2}(e_{1})=\{h_{3}\}, & F_{2}(e_{2})=\{h_{3}\}, \\ 
F_{3}(e_{1})=\{h_{1},h_{3}\}, & F_{3}(e_{2})=\{h_{1},h_{3}\}, \\ 
F_{4}(e_{1})=\{\}, & F_{4}(e_{2})=\{h_{3}\}, \\ 
F_{5}(e_{1})=\{h_{2},h_{3}\}, & F_{5}(e_{2})=\{h_{2}\}, \\ 
F_{6}(e_{1})=\{h_{3}\}, & F_{6}(e_{2})=\{\}, \\ 
F_{7}(e_{1})=\{h_{2},h_{3}\}, & F_{7}(e_{2})=\{h_{2},h_{3}\},%
\end{array}%
\end{equation*}%
and%
\begin{equation*}
\begin{array}{lll}
G_{1}(e_{1})=\{h_{2}\}, &  & G_{1}(e_{2})=\{h_{2}\}, \\ 
G_{2}(e_{1})=\{h_{3}\}, &  & G_{2}(e_{2})=\{h_{3}\}, \\ 
G_{3}(e_{1})=\{h_{2},h_{3}\}, &  & G_{3}(e_{2})=\{h_{2},h_{3}\}, \\ 
G_{4}(e_{1})=\{h_{1},h_{2}\}, &  & G_{4}(e_{2})=\{h_{1},h_{2}\}, \\ 
G_{5}(e_{1})=\{h_{1},h_{3}\}, &  & G_{5}(e_{2})=\{h_{1},h_{3}\}, \\ 
G_{6}(e_{1})=\{h_{1}\}, &  & G_{6}(e_{2})=\{h_{1}\}.%
\end{array}%
\end{equation*}%
Then $\mathcal{T}_{1}$\ and $\mathcal{T}_{2}$ are soft topologies on $X$.
Therefore $(X,\mathcal{T}_{1},\mathcal{T}_{2},E)$ is a \textit{bi-soft
topological space over }$X$. One can easily see that $(X,\mathcal{T}_{1},%
\mathcal{T}_{2},E)$ is a pair-wise \textit{soft} $T_{1}-$\textit{space over }%
$X$.
\end{example}

\begin{proposition}
\label{Prop4}Let $(X,\mathcal{T}_{1},\mathcal{T}_{2},E)$ be a bi-soft
topological space over $X$. Then $(X,\mathcal{T}_{1},\mathcal{T}_{2},E)$\ is
a pair-wise soft $T_{1}-$space if and only if $(X,\mathcal{T}_{1},E)$\ and $%
(X,\mathcal{T}_{2},E)$\ are soft $T_{1}-$spaces.

\begin{proof}
Let $x$,$y\in X$ be such that $x\neq y$. Suppose that $(X,\mathcal{T}_{1},E)$%
\ and $(X,\mathcal{T}_{2},E)$\ are soft $T_{1}-$spaces. Then there exist
some $(F,E)\in \mathcal{T}_{1}$ and $(G,E)\in \mathcal{T}_{2}$\ such that $%
x\in (F,E)$ and$\ y\notin (F,E)$ and $y\in (G,E)$ and $x\notin (G,E)$. In
either case we obtain the requirement and so $(X,\mathcal{T}_{1},\mathcal{T}%
_{2},E)$\ is a pair-wise soft $T_{1}-$space. Conversely we assume that $(X,%
\mathcal{T}_{1},\mathcal{T}_{2},E)$\ is a pair-wise soft $T_{1}-$space. Then
there exist some $(F_{1},E)\in \mathcal{T}_{1}$ and $(G_{1},E)\in \mathcal{T}%
_{2}$\ such that $x\in (F_{1},E)$ and $y\notin (F_{1},E)$ and $y\in
(G_{1},E) $ and $x\notin (G_{1},E)$. Also there exist soft sets $%
(F_{2},E)\in \mathcal{T}_{1}$ and $(G_{2},E)\in \mathcal{T}_{2}$\ such that $%
y\in (F_{2},E)$ and $x\notin (F_{2},E)$ and $x\in (G_{2},E)$ and $y\notin
(G_{2},E)$. Thus $(X,\mathcal{T}_{1},E)$\ and $(X,\mathcal{T}_{2},E)$\ are
soft $T_{1}-$spaces.
\end{proof}
\end{proposition}

\begin{proposition}
\label{Prop3}Let $(X,\mathcal{T}_{1},\mathcal{T}_{2},E)$ be a bi-soft
topological space over $X$. If $(X,\mathcal{T}_{1},\mathcal{T}_{2},E)$\ is a
pair-wise soft $T_{1}-$space then $(X,\mathcal{T}_{1}\vee \mathcal{T}_{2},E)$%
\ is also a soft $T_{1}-$space.

\begin{proof}
Let $x$,$y\in X$ be such that $x\neq y$. There exists $(F,E)\in \mathcal{T}%
_{1}$ such that $x\in (F,E)$,$\ y\notin (F,E)$ and $(G,E)\in \mathcal{T}_{2}$
such that $y\in (G,E)$ and $x\notin (G,E)$. So $(F,E),$ $(G,E)\in \mathcal{T}%
_{1}\vee \mathcal{T}_{2}$ and thus $(X,\mathcal{T}_{1}\vee \mathcal{T}%
_{2},E) $\ is a soft $T_{1}-$space.
\end{proof}
\end{proposition}

\begin{remark}
The converse of Proposition \ref{Prop3}, is not true. This is shown by the
following example:
\end{remark}

\begin{example}
Let $X=\{h_{1},h_{2}\}$, $E=\{e_{1},e_{2}\}$ and 
\begin{eqnarray*}
\mathcal{T}_{1} &=&\{\Phi ,\widetilde{X},(F,E)\}\text{, and} \\
\mathcal{T}_{2} &=&\{\Phi ,\widetilde{X},(G,E)\}\text{,}
\end{eqnarray*}%
where $(F,E),$ and $(G,E)$ are soft sets over $X$, defined as follows:%
\begin{equation*}
\begin{array}{lll}
F(e_{1})=\{h_{1}\}, &  & F(e_{2})=X, \\ 
G(e_{1})=X, &  & G(e_{2})=\{h_{2}\}\text{.}%
\end{array}%
\end{equation*}%
Then $\mathcal{T}_{1}$\ and $\mathcal{T}_{2}$ are soft topologies on $X$.
Therefore $(X,\mathcal{T}_{1},\mathcal{T}_{2},E)$ is a \textit{bi-soft
topological space over }$X$. Both of $(X,\mathcal{T}_{1},E)$ and $(X,%
\mathcal{T}_{2},E)$\ are not soft $T_{1}-$spaces over $X$ and so $(X,%
\mathcal{T}_{1},\mathcal{T}_{2},E)$ is not a pair-wise soft $T_{1}-$space by
Proposition \ref{Prop4}. Now%
\begin{equation*}
\mathcal{T}_{1}\vee \mathcal{T}_{2}=\{\Phi ,\widetilde{X},(F,E),(G,E),(H,E)\}
\end{equation*}%
where%
\begin{equation*}
\begin{array}{lll}
H(e_{1})=\{h_{1}\}, &  & H(e_{2})=\{h_{2}\}\text{.}%
\end{array}%
\end{equation*}%
So $(X,\mathcal{T}_{1}\vee \mathcal{T}_{2},E)$ is a \textit{soft topological
space over }$X$ containing $\mathcal{T}_{1}\cup \mathcal{T}_{2}$. For $%
h_{1},h_{2}\in X$, $(F,E)\in \mathcal{T}_{1}$, $(G,E)\in \mathcal{T}_{2}$
such that 
\begin{equation*}
h_{1}\in (F,E)\text{, }h_{2}\notin (F,E)\text{ and }h_{2}\in (G,E)\text{, }%
h_{1}\notin (G,E)\text{.}
\end{equation*}%
Thus $(X,\mathcal{T}_{1}\vee \mathcal{T}_{2},E)$ is a \textit{soft }$T_{1}-$%
\textit{space over }$X$.
\end{example}

Consider the following example:

\begin{example}
Let $X=\{h_{1},h_{2}\}$, $E=\{e_{1},e_{2}\}$ and 
\begin{eqnarray*}
\mathcal{T}_{1} &=&\{\Phi ,\widetilde{X},(F_{1},E),(F_{2},E),(F_{3},E)\}%
\text{, and} \\
\mathcal{T}_{2} &=&\{\Phi ,\widetilde{X},(G_{1},E),(G_{2},E),(G_{3},E)\}%
\text{,}
\end{eqnarray*}%
where $(F_{1},E),$ $(F_{2},E),$ $(F_{3},E),$ $(G_{1},E),$ $(G_{2},E),$ and $%
(G_{3},E)$ are soft sets over $X$, defined as follows:%
\begin{equation*}
\begin{array}{ll}
F_{1}(e_{1})=\{h_{2}\}, & F_{1}(e_{2})=X, \\ 
F_{2}(e_{1})=X, & F_{2}(e_{2})=\{h_{1}\}, \\ 
F_{3}(e_{1})=\{h_{2}\}, & F_{3}(e_{2})=\{h_{1}\},%
\end{array}%
\end{equation*}%
and%
\begin{equation*}
\begin{array}{ll}
G_{1}(e_{1})=\{h_{1}\}, & G_{1}(e_{2})=X, \\ 
G_{2}(e_{1})=X, & G_{2}(e_{2})=\{h_{2}\}, \\ 
G_{3}(e_{1})=\{h_{1}\}, & G_{3}(e_{2})=\{h_{2}\}\text{.}%
\end{array}%
\end{equation*}%
Then $\mathcal{T}_{1}$\ and $\mathcal{T}_{2}$ are soft topologies on $X$.
Therefore $(X,\mathcal{T}_{1},\mathcal{T}_{2},E)$ is a \textit{bi-soft
topological space over }$X$. Both of $(X,\mathcal{T}_{1},E)$ and $(X,%
\mathcal{T}_{2},E)$\ are soft $T_{1}-$spaces over $X$ and so $(X,\mathcal{T}%
_{1},\mathcal{T}_{2},E)$ is also a pair-wise soft $T_{1}-$space by
Proposition \ref{Prop4}. Now 
\begin{eqnarray*}
\mathcal{T}_{1e_{1}} &=&\{\emptyset ,X,\{h_{2}\}\}, \\
\mathcal{T}_{2e_{1}} &=&\{\emptyset ,X,\{h_{1}\}\},
\end{eqnarray*}%
and%
\begin{eqnarray*}
\mathcal{T}_{2e_{2}} &=&\{\emptyset ,X,\{h_{1}\}\}, \\
\mathcal{T}_{2e_{2}} &=&\{\emptyset ,X,\{h_{2}\}\},
\end{eqnarray*}%
are corresponding parametrized topologies on $X$. Considering the \textit{%
bitopological space} $(X,\mathcal{T}_{1e_{1}},\mathcal{T}_{2e_{1}})$, we see
that $h_{1},h_{2}\in X$ and there do not exist any $\mathcal{T}_{1e_{1}}-$%
open set $U$ such that $h_{1}\in U$, $h_{2}\notin U$ and $\mathcal{T}%
_{2e_{1}}-$open set $V$ such that $h_{2}\in V$, $h_{1}\notin V$. Thus $(X,%
\mathcal{T}_{1e_{1}},\mathcal{T}_{2e_{1}})$ is not a pair-wise $T_{1}-$%
space. Similarly $(X,\mathcal{T}_{1e_{2}},\mathcal{T}_{2e_{2}})$ is not a
pair-wise $T_{1}-$space too.
\end{example}

The following proposition will provide us the condition that will address
this problem when we go for the corresponding parameterized topologies.

\begin{proposition}
Let $(X,\mathcal{T}_{1},\mathcal{T}_{2},E)$ be a bi-soft topological space
over $X$ and $x$,$y\in X$ be such that $x\neq y$. If there exist a $\mathcal{%
T}_{1}-$soft open set $(F,E)$ such that $x\in (F$,$E)$, $y\in (F$,$E)^{c}$
and a $\mathcal{T}_{2}-$soft open set $(G,E)$ such that $y\in (G$,$E)$, $%
x\in (G$,$E)^{c}$, then $(X,\mathcal{T}_{1},\mathcal{T}_{2},E)$ is a
pair-wise soft $T_{1}-$space over $X$ and $(X,\mathcal{T}_{1e},\mathcal{T}%
_{2e})$\ is a pair-wise $T_{1}-$space for each $e\in E$.
\end{proposition}

\begin{proof}
Straightforward.
\end{proof}

\begin{proposition}
Let $(X,\mathcal{T}_{1},\mathcal{T}_{2},E)$ be a bi-soft topological space
over $X$ and $Y$ be a non-empty subset of $X$. If $(X,\mathcal{T}_{1},%
\mathcal{T}_{2},E)$\ is a pair-wise soft $T_{1}-$space then $(Y,\mathcal{T}%
_{1Y},\mathcal{T}_{2Y},E)$\ is also a pair-wise soft $T_{1}-$space.

\begin{proof}
Let $x$,$y\in Y$ be such that $x\neq y$. Then there exist soft sets $%
(F,E)\in \mathcal{T}_{1}$ and $(G,E)\in \mathcal{T}_{2}$ such that $x\in
(F,E)$,$\ y\notin (F,E)$ and $y\in (G,E)$, $x\notin (G,E)$. Now $x\in Y$\
implies that $x\in \widetilde{Y}$. Hence $x\in \widetilde{Y}\cap
(F,E)=(^{Y}F,E)$\ where $(F,E)\in \mathcal{T}_{1}$. Consider $y\notin (F,E)$%
, this means that $y\notin F(e)$ for some $e\in E$.\ Then $y\notin Y\cap
F(e)=Y(e)\cap F(e)$.\ Therefore $y\notin \widetilde{Y}\cap (F,E)=(^{Y}F,E)$.
Similarly it can also be proved that $y\in (G,E)$ and $x\notin (G,E)$
implies that $y\in (^{Y}G,E)$ and $x\notin (^{Y}G,E)$.\ Thus $(Y,\mathcal{T}%
_{1Y},\mathcal{T}_{2Y},E)$\ is also a pair-wise soft $T_{1}-$space.
\end{proof}
\end{proposition}

\begin{proposition}
Every pair-wise soft $T_{1}-$space is also a pair-wise soft $T_{0}-$space.

\begin{proof}
Straightforward.
\end{proof}
\end{proposition}

\begin{example}
Example \ref{Examp2} is a pair-wise soft $T_{0}-$space which is not a
pair-wise soft $T_{1}-$ space over $X$. Another example is given by taking $%
X=\{h_{1},h_{2}\}$, $E=\{e_{1},e_{2}\}$ and 
\begin{eqnarray*}
\mathcal{T}_{1} &=&\{\Phi ,\widetilde{X},(F,E)\}\text{, and} \\
\mathcal{T}_{2} &=&\{\Phi ,\widetilde{X},(G,E)\}\text{,}
\end{eqnarray*}%
where $(F,E),$ and $(G,E)$ are soft sets over $X$, defined as follows:%
\begin{equation*}
\begin{array}{ll}
F(e_{1})=X, & F(e_{2})=\{h_{2}\}, \\ 
G(e_{1})=\{h_{1}\}, & G(e_{2})=X\text{.}%
\end{array}%
\end{equation*}%
Then $\mathcal{T}_{1}$\ and $\mathcal{T}_{2}$ are soft topologies on $X$.
Therefore $(X,\mathcal{T}_{1},\mathcal{T}_{2},E)$ is a \textit{bi-soft
topological space over }$X$. Both of $(X,\mathcal{T}_{1},E)$ and $(X,%
\mathcal{T}_{2},E)$\ are not soft $T_{1}-$spaces over $X$ and so $(X,%
\mathcal{T}_{1},\mathcal{T}_{2},E)$ is not a pair-wise soft $T_{1}-$space by
Proposition \ref{Prop4}, but it is evident that $(X,\mathcal{T}_{1},\mathcal{%
T}_{2},E)$ is a pair-wise soft $T_{0}-$space over $X$.
\end{example}

\begin{definition}
A bi-soft topological space $(X,\mathcal{T}_{1},\mathcal{T}_{2},E)$ over $X$
is said to be pair-wise soft $T_{2}-$space or pair-wise Hausdorff space if
for every pair of distinct points $x$,$y\in X$, there is a $\mathcal{T}_{1}-$%
soft open set $(F,E)$ and a $\mathcal{T}_{2}-$soft open set $(G,E)$ such
that $x\in (F,E)$ and $y\in (G,E)$ and $(F,E)\cap (G,E)=\Phi $.
\end{definition}

\begin{proposition}
Let $(X,\mathcal{T}_{1},\mathcal{T}_{2},E)$ be a bi-soft topological space
over $X$. If $(X,\mathcal{T}_{1},\mathcal{T}_{2},E)$ is a pair-wise soft $%
T_{2}-$space over $X$ then $(X,\mathcal{T}_{1e},\mathcal{T}_{2e})$\ is a
pair-wise $T_{2}-$space for each $e\in E$.
\end{proposition}

\begin{proof}
Suppose that $(X,\mathcal{T}_{1},\mathcal{T}_{2},E)$ is a pair-wise soft $%
T_{2}-$space over $X$. For any $e\in E$%
\begin{eqnarray*}
\mathcal{T}_{1e} &=&\{F(e)\text{ }|\text{ }(F,E)\in \mathcal{T}_{1}\text{ }\}
\\
\mathcal{T}_{2e} &=&\{G(e)\text{ }|\text{ }(G,E)\in \mathcal{T}_{2}\text{ }%
\}.
\end{eqnarray*}%
Let $x$,$y\in X$ be such that $x\neq y$, then there exist $(F,E)\in \mathcal{%
T}_{1}$, $(G,E)\in \mathcal{T}_{2}$ such that%
\begin{equation*}
x\in (F,E)\text{, }y\in (G,E)\text{ and }(F,E)\cap (G,E)=\Phi
\end{equation*}%
This implies that%
\begin{equation*}
x\in F(e)\in \mathcal{T}_{1e}\text{, }y\in G(e)\in \mathcal{T}_{2e}\text{
and }F(e)\cap G(e)=\emptyset \text{.}
\end{equation*}%
Thus $(X,\mathcal{T}_{1e},\mathcal{T}_{2e})$\ is a pair-wise $T_{2}-$space
for each $e\in E$.
\end{proof}

\begin{remark}
Let $(X,\mathcal{T}_{1},\mathcal{T}_{2},E)$ be a pair-wise soft $T_{2}-$%
space over $X$ then $(X,\mathcal{T}_{1},E)$\ and $(X,\mathcal{T}_{2},E)$\
need not be soft $T_{2}-$spaces over $X$.
\end{remark}

\begin{example}
Let $X$ be an infinite set and $E$ be the set of parameters. We define%
\begin{eqnarray*}
\mathcal{T}_{1} &=&\{(F,E)|(F,E)\text{ is a soft set over }X\}\text{ ' Soft
discrete topology over }X\text{'} \\
\mathcal{T}_{2} &=&\{\Phi \}\cup \{(F,E)|(F,E)\text{ is a soft set over }X%
\text{ and }F^{c}(e)\text{ is finite for all }e\in E\}\text{.}
\end{eqnarray*}%
Obviously $\mathcal{T}_{1}$ is a soft topology over $X$. We verify for $%
\mathcal{T}_{2}$\ as:

\begin{enumerate}
\item $\Phi \in \mathcal{T}_{2}$ and $\tilde{X}^{c}=\Phi \Rightarrow $ $%
\tilde{X}\in \mathcal{T}_{2}$.

\item Let $\{(F_{i},E)|$ $i\in I$ $\}$ be a collection of soft sets in $%
\mathcal{T}_{2}$. For any $e\in E$, $F_{i}^{c}(e)$ is finite for all $i\in I$%
\ so that $\underset{i\in I}{\cap }F_{i}^{c}(e)=(\underset{i\in I}{\cup }%
F_{i})^{c}(e)$ is also finite. This means that $\underset{i\in I}{\cup }%
(F_{i},E)\in \mathcal{T}_{2}$.

\item Let $(F,E)$, $(G,E)\in \mathcal{T}_{2}$. Since $F^{c}(e)$, and $%
G^{c}(e)$ are finite sets so as their union $F^{c}(e)\cup G^{c}(e)$. Thus $%
(F\cap G)^{c}(e)$ is finite for all $e\in E$ which shows that $(F,E)\cap
(G,E)\in \mathcal{T}_{2}$.
\end{enumerate}

Then $\mathcal{T}_{1}$ and $\mathcal{T}_{2}$ are soft topologies on $X$. For
any $x,y\in X$ where $x\neq y$, $(x,E)\in \mathcal{T}_{1}$ and $(x,E)^{c}\in 
\mathcal{T}_{2}$ such that 
\begin{equation*}
x\in (x,E),y\in (x,E)^{c}\text{ and }(x,E)\cap (x,E)^{c}=\Phi \text{.}
\end{equation*}%
Thus $(X,\mathcal{T}_{1},\mathcal{T}_{2},E)$ is a pair-wise soft $T_{2}-$%
space over $X$.

Now, we suppose that there are soft sets $(G_{1},E)$, $(G_{2},E)\in \mathcal{%
T}_{2}$ such that 
\begin{equation*}
x\in (G_{1},E),y\in (G_{2},E)\text{ and }(G_{1},E)\cap (G_{2},E)=\Phi \text{.%
}
\end{equation*}%
But then, we must have $(G_{1},E)\tilde{\subset}(G_{2},E)^{c}\Rightarrow
G_{1}(e)\subseteq G_{2}^{c}(e)$ for all $e\in E$, which is not possible
because $G_{1}(e)$ is infinite and $G_{2}^{c}(e)$ is finite. Therefore $(X,%
\mathcal{T}_{2},E)$ is not a soft $T_{2}-$space over $X$.
\end{example}

\begin{remark}
Let $(X,\mathcal{T}_{1},E)$\ and $(X,\mathcal{T}_{2},E)$\ be soft $T_{2}-$%
spaces over $X$ then $(X,\mathcal{T}_{1},\mathcal{T}_{2},E)$ need not be a
pair-wise soft $T_{2}-$space over $X$.
\end{remark}

\begin{example}
Let $X$ be an infinite set and $E$ be the set of parameters. Let $x\neq y$,
where $x,y\in X$, we define%
\begin{gather*}
\mathcal{T}(x)_{1}=\{(F,E)|x\in (F,E)^{c}\text{ is a soft set over }X\}\cup
\{(F,E)|(F,E)\text{ is a } \\
\text{soft set over }X\text{ and }F^{c}(e)\text{ is finite for all }e\in
E\}\} \\
\mathcal{T}(y)_{2}=\{(G,E)|y\in (G,E)\text{ is a soft set over }X\}\cup
\{(G,E)|(G,E)\text{ is a } \\
\text{soft set over }X\text{ and }G^{c}(e)\text{ is finite for all }e\in
E\}\}\text{.}
\end{gather*}%
We verify for $\mathcal{T}(x)_{1}$ as:

\begin{enumerate}
\item $x\notin \Phi \Rightarrow \Phi \in \mathcal{T}(x)_{1}$ and $\tilde{X}%
^{c}=\Phi \Rightarrow $ $\tilde{X}\in \mathcal{T}(x)_{1}$.

\item Let $\{(F_{i},E)|$ $i\in I$ $\}$ be a collection of soft sets in $%
\mathcal{T}(x)_{1}$. We have following three cases:

\begin{description}
\item[(i)] If $x\in (F_{i},E)^{c}$ for all $i\in I$ then $x\in \underset{%
i\in I}{\cap }(F_{i},E)^{c}$ so, in this case, $\underset{i\in I}{\cup }%
(F_{i},E)\in \mathcal{T}(x)_{1}$.

\item[(ii)] If $(F_{i},E)$\ is such that $F_{i}^{c}(e)$ is finite for all $%
e\in E$ so $F_{i}^{c}(e)$ is finite for all $i\in I$\ implies that $\underset%
{i\in I}{\cap }F_{i}^{c}(e)=(\underset{i\in I}{\cup }F_{i})^{c}(e)$ is also
finite. This means that $\underset{i\in I}{\cup }(F_{i},E)\in \mathcal{T}%
(x)_{1}$.

\item[(iii)] If there exist some $j,k\in I$ such that $x\in (F_{j},E)^{c}$
and $F_{k}^{c}(e)$ is finite for all $e\in E$. It means that $\underset{i\in
I}{\cap }F_{i}^{c}(e)(\subset F_{k}^{c}(e))$ is also finite for all $e\in E$
and by definition $\underset{i\in I}{\cup }(F_{i},E)\in \mathcal{T}(x)_{1}$.
\end{description}

\item Let $(F_{1},E)$, $(F_{2},E)\in \mathcal{T}(x)_{1}$. Again we have
following three cases:

\begin{description}
\item[(i)] If $x\in (F_{1},E)^{c}$ and $x\in (F_{2},E)^{c}$ then $x\in
(F_{1},E)^{c}\cup (F_{2},E)^{c}$ and therefore $(F_{1},E)\cap (F_{2},E)\in 
\mathcal{T}(x)_{1}$.

\item[(ii)] If $F_{1}^{c}(e)$ and $F_{2}^{c}(e)$ are finite for all $e\in E$
then their union $F_{1}^{c}(e)\cup F_{2}^{c}(e)$ is also finite. Thus $%
(F_{1}\cap F_{2})^{c}(e)$ is finite for all $e\in E$ which shows that $%
(F_{1},E)\cap (F_{2},E)\in \mathcal{T}(x)_{1}$.

\item[(iii)] If $x\in (F_{1},E)^{c}$ and $F_{2}^{c}(e)$ is finite for all $%
e\in E$ then $x\in F_{1}^{c}(e)\cup F_{2}^{c}(e)=(F_{1}\cap F_{2})^{c}(e)$
and so $x\in ((F_{1},E)\cap (F_{2},E))^{c}$. Thus $(F_{1},E)\cap
(F_{2},E)\in \mathcal{T}(x)_{1}$.
\end{description}
\end{enumerate}

Hence $\mathcal{T}(x)_{1}$ is a soft topology on $X$. For any $p,q\in X$
where $x\neq p$, $x\in (p,E)^{c}\Rightarrow (p,E)\in \mathcal{T}(x)_{1}$ and 
$(p,E)^{c}\in \mathcal{T}(x)_{1}$ such that 
\begin{equation*}
p\in (p,E),q\in (p,E)^{c}\text{ and }(p,E)\cap (p,E)^{c}=\Phi \text{.}
\end{equation*}%
Thus $(X,\mathcal{T}(x)_{1},E)$ is a soft $T_{2}-$space over $X$. Similarly $%
(X,\mathcal{T}(y)_{2},E)$ is a soft $T_{2}-$space over $X$.

Now, $(X,\mathcal{T}(x)_{1},\mathcal{T}(y)_{2},E)$ is a bi-soft topological
space over $X$. For $x,y\in X$, we cannot find any soft sets $(F,E)\in 
\mathcal{T}(x)_{1}$ and $(G,E)\in \mathcal{T}(y)_{2}$ such that%
\begin{equation*}
x\in (F,E),y\in (G,E)\text{ and }(F,E)\cap (G,E)=\Phi
\end{equation*}%
because $y\in (G,E)$ and $(F,E)\cap (G,E)=\Phi $ implies that we must have $%
(F,E)\tilde{\subset}(G,E)^{c}$ which means that $G^{c}(e)$ is finite for all 
$e\in E$ and $F(e)\subseteq G^{c}(e)$ for all $e\in E$, and this is not
possible for $F(e)$ is infinite and $G^{c}(e)$ is finite. Therefore $(X,%
\mathcal{T}(x)_{1},\mathcal{T}(y)_{2},E)\ $is not a pair-wise soft $T_{2}-$%
space over $X$.
\end{example}

\begin{remark}
Let $(X,\mathcal{T}_{1}\vee \mathcal{T}_{2},E)$\ be soft $T_{2}-$space over $%
X$ then $(X,\mathcal{T}_{1},\mathcal{T}_{2},E)$ need not be a pair-wise soft 
$T_{2}-$space over $X$.
\end{remark}

\begin{proposition}
Let $(X,\mathcal{T}_{1},\mathcal{T}_{2},E)$ be a bi-soft topological space
over $X$. If $(X,\mathcal{T}_{1},\mathcal{T}_{2},E)$\ is a pair-wise soft $%
T_{2}-$space then $(X,\mathcal{T}_{1}\vee \mathcal{T}_{2},E)$\ is also a
soft $T_{2}-$space.

\begin{proof}
Let $x$,$y\in X$ be such that $x\neq y$. There exist $(F,E)\in \mathcal{T}%
_{1}$ and $(G,E)\in \mathcal{T}_{2}$ such that $x\in (F,E)$, $y\in (G,E)$
and $(F,E)\cap (G,E)=\Phi $. In either case $(F,E),(G,E)\in \mathcal{T}%
_{1}\vee \mathcal{T}_{2}$.\ Hence $(X,\mathcal{T}_{1}\vee \mathcal{T}_{2},E)$%
\ is a soft $T_{2}-$space over $X$.
\end{proof}
\end{proposition}

\begin{example}
Let $X=\{h_{1},h_{2},h_{3}\}$, $E=\{e_{1},e_{2}\}$ and%
\begin{eqnarray*}
\mathcal{T}_{1} &=&\{\Phi ,\widetilde{X}%
,(F_{1},E),(F_{2},E),(F_{3},E),(F_{4},E)\}\text{, and} \\
\mathcal{T}_{2} &=&\{\Phi ,\widetilde{X},(G_{1},E),(G_{2},E),(G_{3},E)\}%
\text{,}
\end{eqnarray*}%
where $(F_{1},E),$ $(F_{2},E),$ $(F_{3},E),$ $(F_{4},E),$ $(G_{1},E),$ $%
(G_{2},E),$ and $(G_{3},E)$ are soft sets over $X$, defined as follows:%
\begin{equation*}
\begin{array}{ll}
F_{1}(e_{1})=\{h_{1}\}, & F_{1}(e_{2})=\{h_{1}\}, \\ 
F_{2}(e_{1})=\{h_{2}\}, & F_{2}(e_{2})=\{h_{1},h_{2}\}, \\ 
F_{3}(e_{1})=\{\}, & F_{3}(e_{2})=\{h_{1}\}, \\ 
F_{4}(e_{1})=\{h_{1},h_{2}\}, & F_{4}(e_{2})=\{h_{1},h_{2}\},%
\end{array}%
\end{equation*}%
and%
\begin{equation*}
\begin{array}{ll}
G_{1}(e_{1})=\{h_{3}\}, & G_{1}(e_{2})=\{h_{3}\}, \\ 
G_{2}(e_{1})=\{h_{2}\}, & G_{2}(e_{2})=\{h_{2}\}, \\ 
G_{3}(e_{1})=\{h_{2},h_{3}\}, & G_{3}(e_{2})=\{h_{2},h_{3}\}\text{.}%
\end{array}%
\end{equation*}%
Then $\mathcal{T}_{1}$\ and $\mathcal{T}_{2}$ are soft topologies on $X$.
Therefore $(X,\mathcal{T}_{1},\mathcal{T}_{2},E)$ is a \textit{bi-soft
topological space over }$X$. One can easily see that $(X,\mathcal{T}_{1},%
\mathcal{T}_{2},E)$ is not a pair-wise \textit{soft} $T_{2}-$\textit{space
over }$X$ because $h_{1},h_{2}\in X$, and we cannot find any soft sets $%
(F,E)\in \mathcal{T}_{1}$ or $(G,E)\in \mathcal{T}_{2}$ such that 
\begin{equation*}
h_{2}\in (F,E)\text{, }h_{1}\in (G,E)\text{ and }(F,E)\cap (G,E)=\Phi \text{.%
}
\end{equation*}%
Now, we have%
\begin{eqnarray*}
\mathcal{T}_{1}\vee \mathcal{T}_{2} &=&\{\Phi ,\widetilde{X}%
,(F_{1},E),(F_{2},E),(F_{3},E),(F_{4},E),(G_{1},E),(G_{2},E),(G_{3},E), \\
&&(H_{1},E),(H_{2},E),(H_{3},E)\}
\end{eqnarray*}%
where%
\begin{equation*}
\begin{array}{ll}
H_{1}(e_{1})=\{h_{1},h_{3}\}, & H_{1}(e_{2})=\{h_{1},h_{3}\}, \\ 
H_{2}(e_{1})=\{h_{2},h_{3}\}, & H_{2}(e_{2})=X, \\ 
H_{3}(e_{1})=\{h_{3}\}, & H_{3}(e_{2})=\{h_{1},h_{3}\},%
\end{array}%
\end{equation*}%
so $(X,\mathcal{T}_{1}\vee \mathcal{T}_{2},E)$ is a \textit{soft topological
space over }$X$.

For $h_{1},h_{2}\in X$, $(F_{1},E)$, $(G_{2},E)\in \mathcal{T}_{1}\vee 
\mathcal{T}_{2}$ such that 
\begin{equation*}
h_{1}\in (F_{1},E)\text{, }h_{2}\in (G_{2},E)\text{ and }(F_{1},E)\cap
(G_{2},E)=\Phi \text{.}
\end{equation*}%
For $h_{1},h_{3}\in X$, $(F_{1},E),(G_{3},E)\in \mathcal{T}_{1}\vee \mathcal{%
T}_{2}$ such that 
\begin{equation*}
h_{1}\in (F_{1},E)\text{, }h_{3}\in (G_{3},E)\text{ and }(F_{1},E)\cap
(G_{3},E)=\Phi \text{.}
\end{equation*}%
For $h_{2},h_{3}\in X$ and $(G_{2},E),(G_{1},E)\in \mathcal{T}_{1}\vee 
\mathcal{T}_{2}$ such that 
\begin{equation*}
h_{2}\in (G_{2},E)\text{, }h_{3}\in (G_{1},E)\text{ and }(F_{1},E)\cap
(G_{1},E)=\Phi \text{.}
\end{equation*}%
Thus $(X,\mathcal{T}_{1}\vee \mathcal{T}_{2},E)$ is a \textit{soft }$T_{2}-$%
\textit{\ space over }$X$.
\end{example}

\begin{proposition}
Let $(X,\mathcal{T}_{1},\mathcal{T}_{2},E)$ be a bi-soft topological space
over $X$ and $Y$ be a non-empty subset of $X$. If $(X,\mathcal{T}_{1},%
\mathcal{T}_{2},E)$\ is a pair-wise soft $T_{2}-$space then $(Y,\mathcal{T}%
_{1Y},\mathcal{T}_{2Y},E)$\ is also a pair-wise soft $T_{2}-$space.
\end{proposition}

\begin{proof}
Let $x$,$y\in Y$ be such that $x\neq y$. Then there exist soft sets $%
(F,E)\in \mathcal{T}_{1}$ and $(G,E)\in \mathcal{T}_{2}$ such that%
\begin{equation*}
x\in (F,E),y\in (G,E)\text{ and }(F,E)\cap (G,E)=\Phi \text{.}
\end{equation*}%
For each $e\in E$, $x\in F\left( e\right) $, $y\in G\left( e\right) $\ and $%
F\left( e\right) \cap G\left( e\right) =\emptyset $. This implies that $x\in
Y\cap F\left( e\right) $, $y\in Y\cap G(e)$ and 
\begin{eqnarray*}
^{Y}F(e)\cap ^{Y}G(e) &=&(Y\cap F\left( e\right) )\cap (Y\cap G(e)) \\
&=&Y\cap (F(e)\cap G\left( e\right) ) \\
&=&Y\cap \emptyset =\emptyset \text{.}
\end{eqnarray*}%
\ Hence $x\in (^{Y}F,E)\in \mathcal{T}_{1Y},$ $y\in (^{Y}G,E)\in \mathcal{T}%
_{2Y}$\ and $(^{Y}F,E)\cap (^{Y}G,E)=\Phi $ where. Thus $(Y,\mathcal{T}_{1Y},%
\mathcal{T}_{2Y},E)$\ is a pair-wise soft $T_{2}-$space.
\end{proof}

\begin{proposition}
\label{Prop5}Every pair-wise soft $T_{2}-$space is also a pair-wise soft $%
T_{1}-$space.

\begin{proof}
If $(X,\mathcal{T}_{1},\mathcal{T}_{2},E)$\ is a pair-wise soft $T_{2}-$%
space and $x$,$y\in X$ be such that $x\neq y$ then there exist soft sets $%
(F,E)\in \mathcal{T}_{1}$ and $(G,E)\in \mathcal{T}_{2}$ such that%
\begin{equation*}
x\in (F,E),y\in (G,E)\text{ and }(F,E)\cap (G,E)=\Phi \text{.}
\end{equation*}%
As $(F,E)\cap (G,E)=\Phi $, so $x\notin (G,E)$ and$\ y\notin (F,E)$. Hence $%
(X,\mathcal{T}_{1},\mathcal{T}_{2},E)$\ is a pair-wise soft $T_{1}-$space.
\end{proof}
\end{proposition}

\begin{remark}
The converse of Proposition \ref{Prop5} is not true i.e. a pair-wise soft $%
T_{1}-$space need not be a pair-wise soft $T_{2}-$space.
\end{remark}

\begin{example}
The bi-soft topological space $(X,\mathcal{T}_{1},\mathcal{T}_{2},E)$ in
Example \ref{Examp3} is a pair-wise soft $T_{1}-$space over $X$ which is not
a pair-wise \textit{soft Hausdorff space} over $X$.
\end{example}

\begin{remark}
For any \textit{soft set }$(F,E)$ over $X$, $\overline{(F,E)}^{\mathcal{T}}$
will be used to denote the \textit{soft closure}$^{\text{Definition \ref%
{softcl}}}$ of $(F,E)$ with respect to the soft topological space $(X,%
\mathcal{T},E)$ over $X$.
\end{remark}

\begin{theorem}
\label{thm1}Let $(X,\mathcal{T}_{1},\mathcal{T}_{2},E)$ be a bi-soft
topological space over $X$. Then the following are equivalent:

\begin{enumerate}
\item $(X,\mathcal{T}_{1},\mathcal{T}_{2},E)$ be a pair-wise soft Hausdorff
space over $X$.

\item Let $x\in X$, for each point $y$ distinct from $x$, there is a soft
set $(F,E)\in \mathcal{T}_{1}$ such that $x\in (F,E)$ and $y\in \tilde{X}-%
\overline{(F,E)}^{\mathcal{T}_{2}}$.
\end{enumerate}

\begin{proof}
$(1)\Rightarrow (2):$

Suppose that $(X,\mathcal{T}_{1},\mathcal{T}_{2},E)$ is a pair-wise soft
Hausdorff space over $X$ and $x\in X$. For any $y\in X$ such that $y\neq x$,
pair-wise soft Hausdorffness implies that there exist soft sets $F,E)\in 
\mathcal{T}_{1}$ and $(G,E)\in \mathcal{T}_{2}$ such that%
\begin{equation*}
x\in (F,E),y\in (G,E)\text{ and }(F,E)\cap (G,E)=\Phi \text{.}
\end{equation*}%
So that $(F,E)\tilde{\subset}(G,E)^{c}$. Since $\overline{(F,E)}^{\mathcal{T}%
_{2}}$ is the smallest soft closed set in $\mathcal{T}_{2}$ that contains $%
(F,E)$ and $(G,E)^{c}$ is a soft closed set in $\mathcal{T}_{2}$ so $%
\overline{(F,E)}^{\mathcal{T}_{2}}\tilde{\subset}(G,E)^{c}\Rightarrow (G,E)%
\tilde{\subset}(\overline{(F,E)}^{\mathcal{T}_{2}})^{c}$. Thus $y\in (G,E)%
\tilde{\subset}(\overline{(F,E)}^{\mathcal{T}_{2}})^{c}$ or $y\in \tilde{X}-%
\overline{(F,E)}^{\mathcal{T}_{2}}$.

$(2)\Rightarrow (1):$

Let $x,y\in X$ be such that $x\neq y$. By $(2)$ there is a soft set $%
(F,E)\in \mathcal{T}_{1}$ such that $x\in (F,E)$ and $y\in \tilde{X}-%
\overline{(F,E)}^{\mathcal{T}_{2}}$. As $\overline{(F,E)}^{\mathcal{T}_{2}}$
is a soft closed set in $\mathcal{T}_{2}$\ so $(G,E)=\tilde{X}-\overline{%
(F,E)}^{\mathcal{T}_{2}}\in \mathcal{T}_{2}$. Now $x\in (F,E),y\in (G,E)$
and 
\begin{eqnarray*}
(F,E)\cap (G,E) &=&(F,E)\cap (\tilde{X}-\overline{(F,E)}^{\mathcal{T}_{2}})
\\
&\tilde{\subset}&(F,E)\cap (\tilde{X}-(F,E))\text{ \ }\because (F,E)\tilde{%
\subset}\overline{(F,E)}^{\mathcal{T}_{2}} \\
&=&\Phi \text{.}
\end{eqnarray*}%
Thus $(F,E)\cap (G,E)=\Phi $.
\end{proof}
\end{theorem}

\begin{corollary}
\label{cor1}Let $(X,\mathcal{T}_{1},\mathcal{T}_{2},E)$ be a pair-wise soft $%
T_{2}-$space over $X$. Then, for each $x\in X$,%
\begin{equation*}
(x,E)=\dbigcap \left\{ \overline{(F,E)}^{\mathcal{T}_{2}}:x\in (F,E)\in 
\mathcal{T}_{1}\right\} \text{.}
\end{equation*}

\begin{proof}
Let $x\in X$, the existence of a soft open set $x\in (F,E)\in \mathcal{T}%
_{1} $ is guaranteed by pair-wise soft Hausdorffness. If $y\in X$ such that $%
y\neq x$ then, by Theorem \ref{thm1},\ there exists a soft set $(F,E)\in 
\mathcal{T}_{1}$ such that $x\in (F,E)$ and $y\in \tilde{X}-\overline{(F,E)}%
^{\mathcal{T}_{2}}\Rightarrow y\notin (\bar{F}^{\mathcal{T}%
_{2}}(e))\Rightarrow y\notin \underset{x\in (F,E)\in \mathcal{T}_{1}}{%
\dbigcap }(\bar{F}^{\mathcal{T}_{2}}(e))$ for all $e\in E$. Therefore%
\begin{equation*}
\dbigcap \left\{ \overline{(F,E)}^{\mathcal{T}_{2}}:x\in (F,E)\in \mathcal{T}%
_{1}\right\} \tilde{\subset}(x,E)\text{.}
\end{equation*}%
Converse inclusion is obvious as $x\in (F,E)\tilde{\subset}\overline{(F,E)}^{%
\mathcal{T}_{2}}$.
\end{proof}
\end{corollary}

\begin{corollary}
Let $(X,\mathcal{T}_{1},\mathcal{T}_{2},E)$ be a pair-wise soft $T_{2}-$%
space over $X$. Then, for each $x\in X$, $(x,E)^{c}\in \mathcal{T}_{i}$ for $%
i=1,2$.

\begin{proof}
By Corollary \ref{cor1}%
\begin{equation*}
(x,E)^{c}=\dbigcup \left\{ (\overline{(F,E)}^{\mathcal{T}_{2}})^{c}:x\in
(F,E)\in \mathcal{T}_{1}\right\} \text{.}
\end{equation*}%
Since $\overline{(F,E)}^{\mathcal{T}_{2}}$ is a soft closed set in $\mathcal{%
T}_{2}$ so $(\overline{(F,E)}^{\mathcal{T}_{2}})^{c}\in \mathcal{T}_{2}$ and
by the axiom of a soft topological space $\dbigcup \left\{ (\overline{(F,E)}%
^{\mathcal{T}_{2}})^{c}:x\in (F,E)\in \mathcal{T}_{1}\right\} \in \mathcal{T}%
_{2}$. Thus $(x,E)^{c}\in \mathcal{T}_{2}$.

A similar argument holds to show $(x,E)^{c}\in \mathcal{T}_{1}$.
\end{proof}
\end{corollary}

\section{Application of bi-soft topological spaces to rough sets\label{appli}%
}

Rough set theory introduced by Pawlak \cite{Pawlak} is another mathematical
tool to deal with uncertainty. These concepts have been applied successfully
in various fields \cite{Pawlak 2}. In the present paper a new approach for
rough approximations of a soft set is given and some properties of lower and
upper approximations are studied. A bi-soft topological space is applied to
granulate the universe of discourse and a general model of bi-soft
topological spaces based roughness of a soft set is established.

It is easy to see that for any soft set $\left( F,A\right) $ over a set $X$,
the set of parameters $A$ can be extended to $E$ by defining the following
map $\widetilde{F}:E\rightarrow P\left( U\right) $ 
\begin{equation*}
\widetilde{F}\left( e\right) =\left\{ 
\begin{tabular}{ll}
$F\left( a\right) $ & if $a\in A$ \\ 
$\emptyset $ & if $a\in E-A$%
\end{tabular}%
\right.
\end{equation*}%
In the following, a technique is developed to find approximations of a soft
set $\left( F,A\right) $ with respect to a bi-soft topological space $(X,%
\mathcal{T}_{1},\mathcal{T}_{2},E)$. From here onward every soft set $\left(
F,A\right) $ will represented by $(\widetilde{F},E)$.

\begin{definition}
Let $(X,\mathcal{T}_{1},\mathcal{T}_{2},E)$ be a bi-soft topological space
over $X$. Then $(X,\mathcal{T}_{1e},\mathcal{T}_{2e})$ is a \textit{%
bi-topological space} for each $e\in E$. Given a soft set $(\widetilde{F},E)$
over $X$, two soft sets $(\widetilde{\underline{F}}_{\mathcal{T}_{1},%
\mathcal{T}_{2}},E)$ and $(\overline{\widetilde{F}}_{\mathcal{T}_{1},%
\mathcal{T}_{2}},E)$\ are defined as:%
\begin{eqnarray*}
\widetilde{\underline{F}}_{\mathcal{T}_{1},\mathcal{T}_{2}}(e) &=&\mathcal{T}%
_{1e}Int(\widetilde{F}(e))\cap \mathcal{T}_{2e}Int(\widetilde{F}(e)) \\
\overline{\widetilde{F}}_{\mathcal{T}_{1},\mathcal{T}_{2}}(e) &=&\overline{%
\widetilde{F}(e)}^{\mathcal{T}_{1e}}\cup \overline{\widetilde{F}(e)}^{%
\mathcal{T}_{2e}}
\end{eqnarray*}%
for each $e\in E$, where $\mathcal{T}_{ie}Int(\widetilde{F}(e))$ and $%
\overline{\widetilde{F}(e)}^{\mathcal{T}_{ie}}$ denote the interior and
closure of subset $\widetilde{F}(e)$ in the topological space $(X,\mathcal{T}%
_{ie})$ respectively. The soft sets $(\widetilde{\underline{F}}_{\mathcal{T}%
_{1},\mathcal{T}_{2}},E)$ and $(\overline{\widetilde{F}}_{\mathcal{T}_{1},%
\mathcal{T}_{2}},E)$ for called respectively, the lower approximation and
upper approximation of the soft set $(\widetilde{F},E)$ with respect to
bi-soft topological space $(X,\mathcal{T}_{1},\mathcal{T}_{2},E)$ over $X$.
\end{definition}

\begin{definition}
If $(\widetilde{\underline{F}}_{\mathcal{T}_{1},\mathcal{T}_{2}},E)\tilde{=}(%
\overline{\widetilde{F}}_{\mathcal{T}_{1},\mathcal{T}_{2}},E)$ (soft equal)
then the soft set $(\widetilde{F},E)$ is said to be definable and otherwise
it is called a bi-soft topological rough set denoted by the pair $(%
\widetilde{\underline{F}}_{\mathcal{T}_{1},\mathcal{T}_{2}},\overline{%
\widetilde{F}}_{\mathcal{T}_{1},\mathcal{T}_{2}})$. Further,%
\begin{eqnarray*}
pos_{_{\mathcal{T}_{1},\mathcal{T}_{2}}}(\widetilde{F},E) &=&(\widetilde{%
\underline{F}}_{\mathcal{T}_{1},\mathcal{T}_{2}},E); \\
neg_{_{\mathcal{T}_{1},\mathcal{T}_{2}}}(\widetilde{F},E) &=&(\overline{%
\widetilde{F}}_{\mathcal{T}_{1},\mathcal{T}_{2}},E)^{c}; \\
bnd_{_{\mathcal{T}_{1},\mathcal{T}_{2}}}(\widetilde{F},E) &=&(\overline{%
\widetilde{F}}_{\mathcal{T}_{1},\mathcal{T}_{2}},E)-(\widetilde{\underline{F}%
}_{\mathcal{T}_{1},\mathcal{T}_{2}},E).
\end{eqnarray*}%
In order to explain this idea the following example is given:
\end{definition}

\begin{example}
Let $X=\{x_{1},x_{2},x_{3},x_{4},x_{5}\}$ be the set of sample designs of
laptop covers and $E=\{$Red, Green, Blue$\}$ be the set of available colors.
Let us suppose that there are two groups of people. First group consists of $%
3$ members aging $20$, $25$, $28$ and the second group has members aging $35$%
, $42$, $45$. Both groups are asked to select the covers which they approve
according to their likeness and choice. Following the choices they have
made, we obtain $6$ soft sets, given by $(F_{1},E)$, $(F_{2},E)$, $(F_{3},E)$
for the members of first group and $(G_{1},E)$, $(G_{2},E)$, $(G_{3},E)$ for
the members of latter one. Let $\mathcal{T}_{1}$ and $\mathcal{T}_{2}$ be
the soft topologies generated by $(F_{1},E)$, $(F_{2},E)$, $(F_{3},E)$ and $%
(G_{1},E)$, $(G_{2},E)$, $(G_{3},E)$,%
\begin{eqnarray*}
\mathcal{T}_{1} &=&\{\Phi ,\widetilde{X}%
,(F_{1},E),(F_{2},E),(F_{3},E),(F_{4},E),(F_{5},E),(F_{6},E),(F_{7},E),(F_{8},E)\}%
\text{, and} \\
\mathcal{T}_{2} &=&\{\Phi ,\widetilde{X},(G_{1},E),(G_{2},E),(G_{3},E)\}%
\text{,}
\end{eqnarray*}%
where $(F_{1},E),$ $(F_{2},E),$ $(F_{3},E),$ $(F_{4},E),$ $(F_{5},E),$ $%
(F_{6},E),$ $(F_{7},E),$ $(F_{8},E),$ $(G_{1},E),$ $(G_{2},E),$ $(G_{3},E)$
are soft sets over $X$, defined as follows:%
\begin{equation*}
\begin{array}{lll}
F_{1}(Red)=\{x_{2},x_{4}\}, & F_{1}(Green)=\{x_{1},x_{5}\}, & 
F_{1}(Blue)=\{x_{1}\}, \\ 
F_{2}(Red)=\{x_{1},x_{2},x_{4}\}, & F_{2}(Green)=\{x_{1},x_{2},x_{5}\}, & 
F_{2}(Blue)=\{x_{1},x_{3}\}, \\ 
F_{3}(Red)=\{x_{2}\}, & F_{3}(Green)=\{x_{2}\}, & F_{3}(Blue)=\{x_{2}\}, \\ 
F_{4}(Red)=\{x_{2}\}, & F_{4}(Green)=\{x_{2}\}, & F_{4}(Blue)=\emptyset , \\ 
F_{5}(Red)=\{x_{2},x_{4}\}, & F_{5}(Green)=\{x_{1},x_{2},x_{5}\}, & 
F_{5}(Blue)=\{x_{1},x_{2}\}, \\ 
F_{6}(Red)=\{x_{2},x_{4}\}, & F_{6}(Green)=\{x_{1},x_{2},x_{5}\}, & 
F_{6}(Blue)=\{x_{1}\}, \\ 
F_{7}(Red)=\{x_{2}\}, & F_{7}(Green)=\emptyset , & F_{7}(Blue)=\emptyset ,
\\ 
F_{8}(Red)=\{x_{1},x_{2},x_{4}\}, & F_{8}(Green)=\{x_{1},x_{2},x_{5}\}, & 
F_{8}(Blue)=\{x_{1},x_{2},x_{3}\}.%
\end{array}%
\end{equation*}%
and 
\begin{equation*}
\begin{array}{lll}
G_{1}(Red)=\{x_{1},x_{2},x_{4}\}, & G_{1}(Green)=\{x_{2},x_{4},x_{5}\}, & 
G_{1}(Blue)=\{x_{1},x_{2},x_{3}\}, \\ 
G_{2}(Red)=\{x_{2},x_{4}\}, & G_{2}(Green)=\{x_{4}\}, & G_{2}(Blue)=\{x_{2}%
\}, \\ 
G_{3}(Red)=\{x_{1}\}, & G_{3}(Green)=\{x_{2},x_{5}\}, & G_{3}(Blue)=%
\{x_{1},x_{3}\}.%
\end{array}%
\end{equation*}%
Then $\mathcal{T}_{1}$\ and $\mathcal{T}_{2}$ are soft topologies on $X$.
Thus $(X,\mathcal{T}_{1},\mathcal{T}_{2},E)$ is a \textit{bi-soft
topological space over }$X$. We have%
\begin{eqnarray}
\mathcal{T}_{1Red} &=&\{\emptyset
,X,\{x_{2}\},\{x_{2},x_{4}\},\{x_{1},x_{2},x_{4}\}\}, \\
\mathcal{T}_{2Red} &=&\{\emptyset
,X,\{x_{1}\},\{x_{2},x_{4}\},\{x_{1},x_{2},x_{4}\}\},  \notag
\end{eqnarray}%
\begin{eqnarray}
\mathcal{T}_{1Green} &=&\{\emptyset
,X,\{x_{2}\},\{x_{1},x_{2}\},\{x_{1},x_{2},x_{5}\}\}, \\
\mathcal{T}_{2Green} &=&\{\emptyset
,X,\{x_{4}\},\{x_{2},x_{5}\},\{x_{2},x_{4},x_{5}\}\},  \notag
\end{eqnarray}%
\begin{eqnarray}
\mathcal{T}_{2Blue} &=&\{\emptyset
,X,\{x_{1}\},\{x_{2}\},\{x_{1},x_{3}\},\{x_{1},x_{2}\},\{x_{1},x_{2},x_{3}\}%
\}, \\
\mathcal{T}_{2Blue} &=&\{\emptyset
,X,\{x_{2}\},\{x_{1},x_{3}\},\{x_{1},x_{2},x_{3}\}\}.  \notag
\end{eqnarray}%
Consider the soft set $(\widetilde{F},E)$ over $X$ that describes the choice
of a random customer Mr. X, whose age is in the range of\ $20-45$, where%
\begin{equation*}
\begin{array}{lll}
\widetilde{F}(Red)=\{x_{2},x_{4},x_{5}\}, & \widetilde{F}(Green)=\emptyset ,
& \widetilde{F}(Blue)=\{x_{1},x_{3},x_{4}\}.%
\end{array}%
\end{equation*}%
The lower approximation $(\widetilde{\underline{F}}_{\mathcal{T}_{1},%
\mathcal{T}_{2}},E)$ and upper approximation $(\overline{\widetilde{F}}_{%
\mathcal{T}_{1},\mathcal{T}_{2}},E)$ of the soft set $(\widetilde{F},E)$
with respect to bi-soft topological space $(X,\mathcal{T}_{1},\mathcal{T}%
_{2},E)$ over $X$ is computed as:%
\begin{equation*}
\begin{array}{lll}
\widetilde{\underline{F}}_{\mathcal{T}_{1},\mathcal{T}_{2}}(Red)=%
\{x_{2},x_{4}\}, & \widetilde{\underline{F}}_{\mathcal{T}_{1},\mathcal{T}%
_{2}}(Green)=\emptyset , & \widetilde{\underline{F}}_{\mathcal{T}_{1},%
\mathcal{T}_{2}}(Blue)=\{x_{1},x_{3}\}, \\ 
\overline{\widetilde{F}}_{\mathcal{T}_{1},\mathcal{T}_{2}}(Red)=X, & 
\overline{\widetilde{F}}_{\mathcal{T}_{1},\mathcal{T}_{2}}(Green)=\emptyset ,
& \overline{\widetilde{F}}_{\mathcal{T}_{1},\mathcal{T}_{2}}(Blue)=%
\{x_{1},x_{3},x_{4},x_{5}\}.%
\end{array}%
\end{equation*}%
Thus%
\begin{eqnarray*}
pos_{_{\mathcal{T}_{1},\mathcal{T}_{2}}}(\widetilde{F},E) &=&(\widetilde{%
\underline{F}}_{\mathcal{T}_{1},\mathcal{T}_{2}},E); \\
&&%
\begin{array}{l}
\widetilde{\underline{F}}_{\mathcal{T}_{1},\mathcal{T}_{2}}(Red)=%
\{x_{2},x_{4}\}, \\ 
\widetilde{\underline{F}}_{\mathcal{T}_{1},\mathcal{T}_{2}}(Green)=\emptyset
, \\ 
\widetilde{\underline{F}}_{\mathcal{T}_{1},\mathcal{T}_{2}}(Blue)=%
\{x_{1},x_{3}\},%
\end{array}
\\
neg_{_{\mathcal{T}_{1},\mathcal{T}_{2}}}(\widetilde{F},E) &=&(\overline{%
\widetilde{F}}_{\mathcal{T}_{1},\mathcal{T}_{2}},E)^{c}; \\
&&%
\begin{array}{l}
\overline{\widetilde{F}}_{\mathcal{T}_{1},\mathcal{T}_{2}}^{c}(Red)=%
\emptyset , \\ 
\overline{\widetilde{F}}_{\mathcal{T}_{1},\mathcal{T}_{2}}^{c}(Green)=X, \\ 
\overline{\widetilde{F}}_{\mathcal{T}_{1},\mathcal{T}_{2}}^{c}(Blue)=\{x_{2}%
\},%
\end{array}
\\
bnd_{_{\mathcal{T}_{1},\mathcal{T}_{2}}}(\widetilde{F},E) &=&(\overline{%
\widetilde{F}}_{\mathcal{T}_{1},\mathcal{T}_{2}},E)-(\widetilde{\underline{F}%
}_{\mathcal{T}_{1},\mathcal{T}_{2}},E); \\
&&%
\begin{array}{l}
(\overline{\widetilde{F}}_{\mathcal{T}_{1},\mathcal{T}_{2}}-\widetilde{%
\underline{F}}_{\mathcal{T}_{1},\mathcal{T}_{2}})(Red)=\{x_{1},x_{3},x_{5}\},
\\ 
(\overline{\widetilde{F}}_{\mathcal{T}_{1},\mathcal{T}_{2}}-\widetilde{%
\underline{F}}_{\mathcal{T}_{1},\mathcal{T}_{2}})(Green)=\emptyset , \\ 
(\overline{\widetilde{F}}_{\mathcal{T}_{1},\mathcal{T}_{2}}-\widetilde{%
\underline{F}}_{\mathcal{T}_{1},\mathcal{T}_{2}})(Blue)=\{x_{4},x_{5}\},%
\end{array}%
\end{eqnarray*}
From these approximations, we may assert the following:

\begin{enumerate}
\item For Mr. X, $x_{2}$ or $x_{4}$ will be best choices in red color, no
design will be selected in green color and, $x_{1}$ and $x_{3}$ should be
preferred if the color choice is blue.

\item No design in red color can be considered as a bad choice, no design
can be selected in green color and, $x_{2}$ should not be selected in blue
color.

\item $x_{1}$, $x_{3}$, $x_{5}$ may be chosen in red color and, $x_{4}$ and $%
x_{5}$ are also considerable in case the color choice is blue.
\end{enumerate}
\end{example}

\begin{theorem}
Let $(X,\mathcal{T}_{1},\mathcal{T}_{2},E)$ be a bi-soft topological space
and $\left( \widetilde{F},E\right) $, $\left( \widetilde{F}_{1},E\right) $, $%
\left( \widetilde{F}_{2},E\right) $ be soft sets over $X$. Then:%
\begin{equation*}
\begin{tabular}{ll}
$1$ & $(\widetilde{\underline{F}}_{\mathcal{T}_{1},\mathcal{T}_{2}},E)\tilde{%
\subset}(\widetilde{F},E)\tilde{\subset}(\overline{\widetilde{F}}_{\mathcal{T%
}_{1},\mathcal{T}_{2}},E),$ \\ 
$2$ & $\underline{\Phi }_{\mathcal{T}_{1},\mathcal{T}_{2}}\tilde{=}\Phi 
\tilde{=}\overline{\Phi }_{\mathcal{T}_{1},\mathcal{T}_{2}},$ \\ 
$3$ & $\underline{\widetilde{X}}_{\mathcal{T}_{1},\mathcal{T}_{2}}\tilde{=}%
\widetilde{X}\tilde{=}\overline{\widetilde{X}}_{\mathcal{T}_{1},\mathcal{T}%
_{2}},$ \\ 
$4$ & $(\widetilde{\underline{F}}_{\mathcal{T}_{1},\mathcal{T}_{2}},E)\cap (%
\widetilde{\underline{F}}_{\mathcal{T}_{1},\mathcal{T}_{2}},E)_{\mathcal{T}%
_{1},\mathcal{T}_{2}}\tilde{=}(\widetilde{\underline{F}}_{\mathcal{T}_{1},%
\mathcal{T}_{2}},E)$ \\ 
$5$ & $(\overline{\widetilde{F}_{1}\cap \widetilde{F}_{2}}_{_{\mathcal{T}%
_{1},\mathcal{T}_{2}}},E)\tilde{\subset}(\overline{\widetilde{F}_{1}}_{%
\mathcal{T}_{1},\mathcal{T}_{2}},E)\cap (\overline{\widetilde{F}_{2}}_{%
\mathcal{T}_{1},\mathcal{T}_{2}},E),$ \\ 
$6$ & $(\widetilde{\underline{F}}_{\mathcal{T}_{1},\mathcal{T}_{2}},E)\cup (%
\widetilde{\underline{F}}_{\mathcal{T}_{1},\mathcal{T}_{2}},E)_{\mathcal{T}%
_{1},\mathcal{T}_{2}}\tilde{\subset}(\widetilde{\underline{F}}_{\mathcal{T}%
_{1},\mathcal{T}_{2}},E),$ \\ 
$7$ & $(\overline{\widetilde{F}_{1}\cup \widetilde{F}_{2}}_{_{\mathcal{T}%
_{1},\mathcal{T}_{2}}},E)\tilde{=}(\overline{\widetilde{F}_{1}}_{\mathcal{T}%
_{1},\mathcal{T}_{2}},E)\cup (\overline{\widetilde{F}_{2}}_{\mathcal{T}_{1},%
\mathcal{T}_{2}},E),$ \\ 
$8$ & $(\widetilde{F}_{1},E)\tilde{\subset}(\widetilde{F}_{2},E)\Rightarrow (%
\widetilde{\underline{F}}_{_{\mathcal{T}_{1},\mathcal{T}_{2}}},E)\tilde{%
\subset}(\widetilde{\underline{F}}_{_{\mathcal{T}_{1},\mathcal{T}_{2}}},E),(%
\overline{\widetilde{F}_{1}}_{\mathcal{T}_{1},\mathcal{T}_{2}},E)\tilde{%
\subset}(\overline{\widetilde{F}_{2}}_{\mathcal{T}_{1},\mathcal{T}_{2}},E),$
\\ 
$9$ & $(\overline{\widetilde{F}}_{\mathcal{T}_{1},\mathcal{T}_{2}},E)\tilde{=%
}((\widetilde{\underline{F}}_{_{\mathcal{T}_{1},\mathcal{T}%
_{2}}},E)^{c})^{c},$ \\ 
$10$ & $((\overline{\widetilde{F}}_{\mathcal{T}_{1},\mathcal{T}%
_{2}},E)^{c})^{c}\tilde{=}(\widetilde{\underline{F}}_{_{\mathcal{T}_{1},%
\mathcal{T}_{2}}},E),$ \\ 
$11$ & $(\underline{(\underline{F}_{_{\mathcal{T}_{1},\mathcal{T}_{2}}})}_{_{%
\mathcal{T}_{1},\mathcal{T}_{2}}},E)\tilde{=}(\widetilde{\underline{F}}_{_{%
\mathcal{T}_{1},\mathcal{T}_{2}}},E),$ \\ 
$12$ & $(\overline{(\overline{\widetilde{F}}_{_{\mathcal{T}_{1},\mathcal{T}%
_{2}}})}_{_{\mathcal{T}_{1},\mathcal{T}_{2}}},E)\tilde{=}(\overline{%
\widetilde{F}}_{\mathcal{T}_{1},\mathcal{T}_{2}},E).$%
\end{tabular}%
\end{equation*}
\end{theorem}

\begin{conclusion}
The concept of soft topological spaces is generalized to bi-soft topological
spaces. Some basic notions and the inter-relations of classical and
generalized concepts have been studied in detail. It is worth mentioning
that the purpose of this paper is just to initiate the concept, and there is
a lot of scope for the researchers to make their investigations in this
field. This is a beginning of some new generalized structure and the concept
of separation axioms may be studied further for regular and normal bi-soft
topological spaces that is our next goal too. The topologies induced by the
definitions of information systems through rough sets$^{\text{\cite{Pawlak}}%
} $ give rise to a natural bitopological space on initial universal set and
this fact increases the interest in bitopological environment. The
connections of information systems and soft sets can be another key to
search for the structure of bi-soft topological spaces in real world
phenomenon. Separation axioms have an application in digital topology$^{%
\text{\cite{Kong}}}$ and a hybrid generalization of the axioms may also be
of some use there.
\end{conclusion}

\begin{summary}
The results on implications are summarized in Figure \ref{Figure 1}:
\end{summary}

\FRAME{ftbpFU}{2.6013in}{3.0407in}{0pt}{\Qcb{Summary of Results}}{\Qlb{%
Figure 1}}{Figure}{\special{language "Scientific Word";type
"GRAPHIC";maintain-aspect-ratio TRUE;display "USEDEF";valid_file "F";width
2.6013in;height 3.0407in;depth 0pt;original-width 2.5598in;original-height
2.9966in;cropleft "0";croptop "1";cropright "1";cropbottom "0";filename
'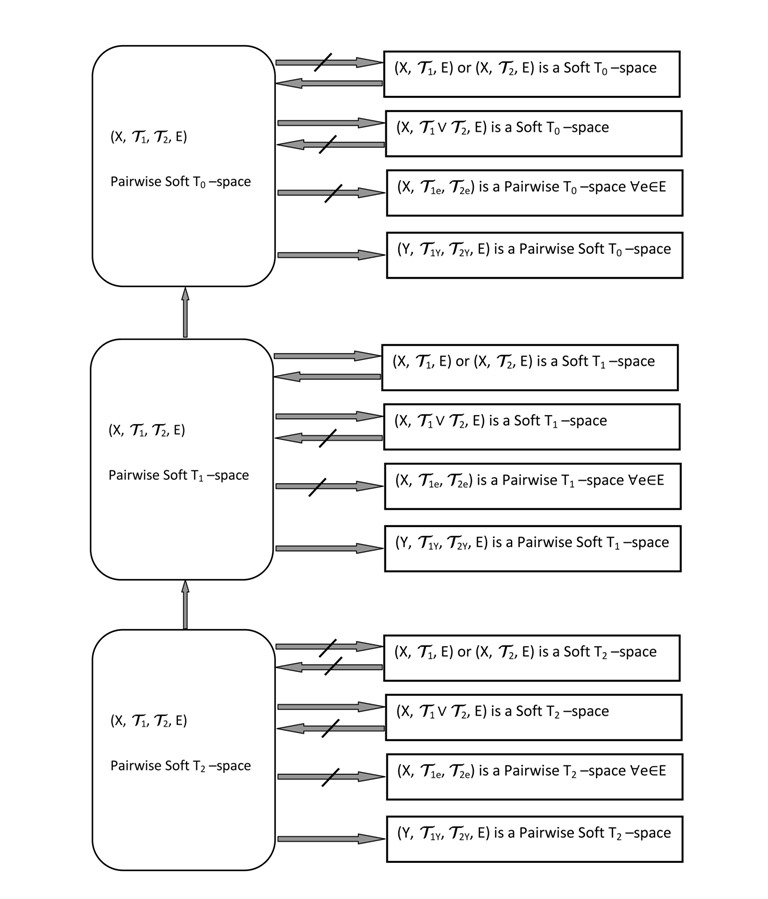';file-properties "XNPEU";}}


\begin{thebibliography}{99}
\bibitem{Ali 1} M.I. Ali, F. Feng, X.Y. Liu, W.K. Min, M. Shabir, On some
new operations in soft set theory, Comput. Math. Appl. 57 (2009) 1547--1553.

\bibitem{Ali 2} M. I. Ali, M. Shabir, M. Naz, Algebraic structures of soft
sets associated with new operations, Comput. Math. Appl. 61 (2011)
2647--2654.

\bibitem{Cagman} N. Cagman, S. Karatas and S. Enginoglu, Soft topology,
Comput. Math. Appl. 62 (2011), 351-358.

\bibitem{Bashir} S. Hussain and B. Ahmad, Some properties of soft
topological spaces, Comput. Math. Appl. 62 (2011) 4058-4067.

\bibitem{Kelly} J. C. Kelly, Bitopological spaces, Proc. London Math. Soc.
13 (1963), 71-89.

\bibitem{Kim} Y. W. Kim, Pairwise compactness, Publ. Math. 15 (1968), 87-90.

\bibitem{Kong} T.Y. Kong, R. Kopperman, P.R. Meyer, A topological approach
to digital topology, Amer. Math. Monthly, 98 (1991), 901-917.

\bibitem{Lal} S. Lal, Pairwise concepts in bitopological spaces, J. Aust.
Math. Soc. (Ser.A), 26 (1978), 241-250.

\bibitem{Lane} E. P. Lane, Bitopological spaces and quasi-uniform spaces,
Proc. London Math. Soc. 17 (1967), 241-256.

\bibitem{Maji 1} P.K. Maji, A. R. Roy, R. Biswas, An application of soft
sets in a decision making problem, Comput. Math. Appl. 44 (2002) 1077 1083.

\bibitem{Maji 2} P.K. Maji, R. Biswas, A.R. Roy, Soft set theory, Comput.
Math. Appl. 45 (2003) 555-562.

\bibitem{Min} W. K. Min, A note on soft topological spaces, Comput. Math.
Appl. 62 (2011), 3524-3528.

\bibitem{Molod} D. Molodtsov, Soft set theory first results, Comput. Math.
Appl. 37 (1999) 19-31.

\bibitem{Murdesh} M. G. Murdeshwar and S. A. Naimpally, Quasi-uniform
Topological Spaces, Monograph Noordhoof Ltd (1966).

\bibitem{Patty} C. W. Patty, Bitopological spaces, Duke Math. J. 34 (1967),
387-392.

\bibitem{Pawlak} Z. Pawlak, Rough sets, Int. J. Comp. and Inf. Sci.,
11(1982) 341-356.

\bibitem{Pawlak 2} Z. Pawlak, Rough Sets, Theoretical Aspects of Reasoning
About Data, Kluwer Academic Publishers, (1991).

\bibitem{Pervin} W. J. Pervin, Connectedness in bitopological spaces, Indag.
Math. 29 (1967), 369-372.

\bibitem{Pie} D. Pie, D. Miao, From soft sets to information systems, Granu.
comput. IEEE Inter. Conf. 2 (2005), 617--621.

\bibitem{Reilly} I. L. Reilly, On bitopological separation properties, Nanta
Math. 5 (1972), 14-25.

\bibitem{Singal} M.K. Singal, Asha Rani Singal, Some more separation axioms
in bitopological spaces, Ann. Soc. Sci. Bruxelles, 84 (1970), 207-230

\bibitem{Shabir} M. Shabir, M. Naz, On soft topological spaces, Comput.
Math. Appl. 61(2011) 1786--1799.
\end{thebibliography}
\end{document}